\documentclass{article}
\usepackage{graphicx} % Required for inserting images
\usepackage{amsmath}
\usepackage[utf8]{inputenc} 

\usepackage{yfonts}
\usepackage{lmodern}
\usepackage{amsthm}
\usepackage[T1]{fontenc}
\usepackage[english,activeacute]{babel}
\usepackage{mathtools}
\usepackage{amssymb}
\usepackage{ mathrsfs }
\usepackage{tikz}

\usepackage{graphicx}
\graphicspath{ {C:\Users\HP\Desktop\Master\Doctorado\Articulo1} }
\usepackage{float}
\usepackage[rightcaption]{sidecap}
\usepackage{subfig} 

\usepackage[
a4paper,
textwidth=80ex,
]{geometry}

\newtheorem{definition}{Definition}%[chapter]
\newtheorem{prop}[definition]{Proposition}%[chapter]\left\lbrace
\newtheorem{theorem}[definition]{Theorem}%[chapter]
\newtheorem{lemma}[definition]{Lema}%[chapter]
\newtheorem{remark}[definition]{Remark}%[chapter]

\newcommand{\txtd}{\mathrm{d}}

\title{Mean-Field Limits for Stochastic Interacting Particles on Digraph Measures}
\author{Christian Kuehn and Carlos Pulido}
\date{\today}

\begin{document}

\maketitle
    \begin{abstract}
    Many natural phenomena are effectively described by interacting particle systems, which can be modeled using either deterministic or stochastic differential equations (SDEs). In this study, we specifically investigate particle systems modeled by SDEs, wherein the mean field limit converges to a Vlasov-Fokker-Planck-type equation. Departing from conventional approaches in stochastic analysis, we explore the network connectivity between particles using diagraph measures (DGMs). DGMs are one possible tool to capture sparse, intermediate and dense network/graph interactions in the mean-field thereby going beyond more classical approaches such as graphons. Since the main goal is to capture large classes of mean-field limits, we set up our approach using measure-theoretic arguments and combine them with suitable moment estimates to ensure approximation results for the mean-field.
    \end{abstract}

	\section{Introduction}

		Consider a physical system composed of $N$ particles that interact with each other through a network/graph of connections, where information about the connectivity between particles can be represented by a graph $G$. The vertices of this graph represent each particle, and the edges of the graph model pairwise interactions between them.
		
		The study of these systems, particularly when $N$ is a very large number, has been approached in various ways, depending on different mathematical treatments used to represent the graph. To set the context, let us begin by introducing the equations that determine the system's dynamics. Let $X_i(t)\in\mathbb{R}^d$ denote the state of particle $i$, and consider the following stochastic differential equations (SDEs) that define the evolution of all particles in the system.
	\begin{align}
		&\txtd X_i(t) = f(X_i(t)) \txtd t + \frac{1}{N} \sum_{j=1}^{N} A^{N}_{ij} g(X_i, X_j) \txtd t + \frac{1}{N} \sum_{j=1}^{N} \hat{A}^{N}_{ij} h(X_i, X_j) \txtd B_t^i, \nonumber\\
		&X_i(0)=X_i^0,\quad i = 1, \ldots, N.
		\label{eq1}
	\end{align}
	Here, $A_{i,j}^N$ is the adjacency matrix of the system, representing the basic deterministic network interaction between particles $i$ and $j$, and $\hat{A}^{N}_{ij}$ serves as an adjacency matrix to represent, how interactions between particles $i$ and $j$ influence the noise term. ${A}^{N}_{ij}$ and $\hat{A}^{N}_{ij}$ are not necessarily identical, which already provides a useful generalization to existing mean-fields that aim to include network coupling. The functions $f$, $g$ and $h$ are Lipschitz and bounded. The noise term contains a sequence $\left\lbrace B^i_t \right\rbrace $, for $1 \leq i \leq N$, of independent and identically distributed Brownian motions in $\mathbb{X}\subseteq \mathbb{R}^d$. The initial conditions $\left\lbrace X_i^0 \right\rbrace_{i=1,..,N}$ form a sample of independently and identically distributed random variables with a probability distribution $\mu^{0}\in\mathcal{P}(\mathbb{X})$. $X_i(t)$ represents, at the microscopic level, the trajectory of particle $i$. As the number of particles in the system becomes very large, the study of the system's dynamics becomes increasingly complex. Therefore, to analyze these issues, various techniques have been introduced, focusing on the study of the density of a typical particle and its evolution~\cite{BraunHepp,Dobrushin1979,Neunzert,Sznitman,Golse}. These techniques were classically developed on an all-to-all coupled graphs.

	The main objective of this work is to determine the mean-field limit for the interacting particle system~\eqref{eq1}. The main challenge is that we have to consider also a suitable notion of graph limits and incorporate the graph limit object in our analysis. For dense graphs, it is understood that graphons provide a suitable tool, and one can stay within a familiar setting of functions and norms as graphons are just functions themselves. Yet, to capture sparse and intermediate density graph limits, it is necessary to work with graph limits that are only measures. In this work we show that mean-fields can be obtained on a type of graph representation called diagraph measure (DGM).
	
	The study of the mean-field limit for equations similar to \eqref{eq1} has been approached in different ways, depending on how the adjacency matrices have been handled in the analysis.
	
	First, in cases where the stochastic terms absent, numerous studies explore the mean-field limit by considering $A_{ij}=1$, which are considered mean-field limits for systems of completely exchangeable particles, as studied in e.g.~ \cite{Jabin,BraunHepp,Dobrushin1979,Neunzert}. On the other hand, when considering systems of non-exchangeable particles and accounting for graph properties, the mean-field limit has been derived in various frameworks, e.g., in~\cite{medvedev,MedvedevGeorgi,GkogkasKuehn,Kuehn,Poyato,AyiDuteil,Paul2022}.
Each one employs different graph representation techniques such as graphons \cite{Poyato,MedvedevGeorgi,medvedev}, diagraph measures \cite{Kuehn}, or graphops \cite{GkogkasKuehn}. In all these cases, it is demonstrated that in the limit, the system satisfies a Vlasov equation.
	
	Second, when considering the stochastic terms, different perspectives have been studied. For example, \cite{BreschJabinSoler,Bayraktar,Coppini} have addressed various treatments of the matrix, ranging from $A_{ij}=1$ to considering the graph as a graphon. Here we focus on studying the graph as a diagraph measure (DGM) and derive the resulting Vlasov-Fokker-Plank equation in the stochastic case. To approach this, we will concentrate on the coupling method~\cite{ChaintronReview,ChaintronReview2}, and utilize the techniques of Sznitmann \cite{Sznitman}, along with those used in \cite{Bayraktar,Coppini}, to study the mean-field limit of the particle system \eqref{eq1}.
	
	As our goal is to obtain a mean-field limit, we aim to determine a measure of the form $\mu(t,x)$ representing the mesoscopic/typical particle evolution in time and space of the original microscopic model. Since the SDEs for the particles lead to a stochastic process, the density $\mu$ must be a probability measure. However, as we are working with DGMs, this measure will also depend on an additional variable representing the heterogeneity of the graph, which we denote as $u \in I$, where $I$ is the set of all possible values that the graph variable can take. Our probability measure will take the form $\mu(t,x,u)$, and we will define $\bar{\mu}(t,x)=\int_I \mu(t,x,u)  ~\txtd u$. For simplicity, we will also use the following notation: $\mu(t,x) = \mu_t(x)$ and $\mu(t,x,u) = \mu_{u,t}(x)$.
	
	Our objective is, therefore, to establish, from \eqref{eq1}, a probability measure whose limit converges in a suitable metric space to a (weak) solution of some Vlasov-Fokker-Planck equation. To achieve this, we will represent the law of the entire system through the empirical measure, which will play a central role. The empirical measure is defined as follows:
	$$\mu^N_{t} = \frac{1}{N} \sum_{i=1}^N \delta_{X_i(t)}, \text{ for } t \in [0,T].$$
	
	The next step, once the empirical measure is defined, is to identify a problem with sufficiently good properties, whose solution can be compared with ours. This is the basis of coupling methods. We seek a problem where the dynamics of each particle is suitably independent of the others, and we are going to prove the existence of a solution to this problem. Once the existence results is established, we are then going to show that as the number of particles tends to infinity, both problems exhibit similar behavior. Mathematically, this is expressed by proving that the distances between the empirical measure and a suitable limiting measure tend to zero, as $N\rightarrow\infty$, in some metric space on probability measures that we have to select.
 
    To start, let us introduce the system of independent processes that we are going to consider:
	\begin{equation}
		\begin{split}
			X_u(t) =& X_u(0) + \int_{0}^{t} f(X_u(s)) ~\txtd  s + \int_{0}^t \int_I \int_{\mathbb{X}} g(X_u(s), y) ~\mu_{v, s}(\txtd y)\eta^u(\txtd v) \txtd s \\
			&+ \int_{0}^t \int_I \int_{\mathbb{X}} h(X_u(s), y)~\mu_{v, s}(\txtd y)\hat{\eta}^u(\txtd v) \txtd B_s^u,
		\end{split}
		\label{indep}
	\end{equation}
	where $\mu_{u, t} = \mathcal{L}(X(t)|U = u)$, and $U$ is a uniform random variable on $I$. The initial conditions $X_u(0)=X_u^0$, for $u\in I$, form a sample of independently and identically distributed random variables with a probability distribution $\bar{\mu}^0\in\mathcal{P}(\mathbb{X})$. Note that two new measures appear, which are finite positive Borel measures $\eta,\hat{\eta}\in \mathcal{B}(I, \mathcal{M}_+(I))$. The measures $\eta,\hat{\eta}$ are the fiber measures associated to the adjacency matrices of the graphs $A^N_{i,j}$ and $\hat{A}^N_{i,j}$. The fiber measures are defined on a vertex space $I$. In this way, these measures represents the edges between different $u$ and the rest of the vertices of $I$. This representation of graphs by a measure is known as a diagraph measure (DGM); the concept is a step towards potentially covering even broader classes of graph limits, e.g., those given by graphops (graph operators). For background on graph limits and DGMs we refer to~\cite{BackhauszSzegedy,Kuehn}. The Brownian motion satisfies the same properties as before.
	
	Next, observe in \eqref{indep} that the particle dynamics, for a given $u$, are independent of the dynamics of the other particles. Therefore, the system represents an independent particle process. Furthermore, the graph's heterogeneity variable is taken as a random variable, uniformly distributed over the interval $I$. Similarly to the previous section, we can define the probability measure associated with the solution $X$ as $\bar{\mu} = \int_I \mu_{u,t} \, \textnormal{d}u$. Since the aim was to compare the two systems as the number of particles becomes sufficiently large, we want to determine whether there is a metric space with metric $d$ such that:
	$$d(\bar{\mu}, \mu_N) \rightarrow 0, \text{ as } N \rightarrow \infty.$$
	Since we are dealing with graphs, the graphs $A^N_{i,j}$ and $\hat{A}^N_{i,j}$  will also change with the number of particles, and we must also be able to represent their convergence as $N$ tends to infinity. Moreover, as we are comparing it with the independent system, we must employ a technique of representing matrices/operators in the form of measures and determine a space in which we can establish the convergence.
 
 In this paper, we will focus on proving that if our two graphs, each represented by a directed graph measure (DGM), converge towards certain measures, then under specific assumptions (Assumptions $H$ and $\tilde{H}$ below), the empirical measure describing the dynamics of the states of the particles will converge to a probability measure $\mu_{u,t}$. This limit satisfies the following Vlasov-Fokker-Plank equation:

\begin{equation}
	\begin{split}
		&\partial_t \bar{\mu}_{t} + \partial_x \left(\bar{\mu}_{t} f(x) + \int_I{\mu}_{u,t} \int_I \int_{\mathbb{X}} g(x,y) ~\mu_{v,t}(\text{d}y) \eta^u(\text{d}v)\text{d}u\right) \\
		&+ \frac{1}{2}\partial_x^2\left(\int_I {\mu}_{u,t} \left[\int_I \int_{\mathbb{X}} h(x,y) ~\mu_{v,t}(\text{d}y) \hat{\eta}^u(\text{d}v)\right]^2~\text{d}u\right)=0.
	\end{split}
\end{equation}

This will be proved in Theorems \ref{teorema} and \ref{teorema2}, where in each theorem, we will handle the limit measure $\mu_{u,t}$ differently. However, in both cases, we will obtain the same Vlasov-Fokker-Plank equation.

The paper is organized as follows. In Section 2, we introduce various measure and probability spaces, along with the associated metrics, that will serve as the foundation for establishing the convergence of the empirical measure and of DGMs. Section 3 states the key results and interpretations of the work, along with the necessary assumptions for proving these results. The proof of Theorem \ref{teorema} is presented in Section 4, where we focus on studying the independent process \eqref{indep}. The existence and uniqueness of the solution to this problem will be explored in the same section. In Section 5, we employ coupling methods to illustrate the convergence in law of the empirical measure towards the solution of \eqref{indep}, thereby proving Theorem \ref{teorema}. Section 6 is dedicated to proving Theorem \ref{teorema2} using techniques similar to those employed in the preceding sections. Finally, Section 7 covers examples of Vlasov-Fokker-Plank equations for some DGMs.
	
	\section{Metric Measure Spaces and Digraph Measures}
	
		We are going to introduce various spaces and notation that we will use.
	We are concerned with the dynamics of particles in a finite time interval, denoted as $[0, T]$, where we fix $T>0$ as a parameter. Our analysis takes place within the framework of a filtered probability space $(\Omega, \mathcal{F}, \{F_t\}_{t\in[0,T]}, \mathbb{P})$, where $\{F_t\}$ denotes a filtration that complies with standard conditions. The particle trajectories $X(t)$ take values in $\mathbb{X}\subseteq\mathbb{R}^d$. 
	
	First of all, let us determine the measurement space over which we are going to define our digraphs. Let $I$ be a complete metric space, and consider the space $\mathcal{B}(I, \mathcal{M}_+(I))$, which is the space of bounded measurable functions from $I$ to $\mathcal{M}_+(I)$, where $\mathcal{M}_+(I)$ is the set of all finite Borel positive measures on $I$. Also, consider $\mathcal{C}(I, \mathcal{M}_+(I))$ as the space of continuous functions from $I$ to $\mathcal{M}_+(I$). Let us define $\mathcal{B}(\mathbb{X})$ the space of bounded measurable functions from $\mathbb{X}$ to $\mathbb{R}^d$ and 
   $$L_f:=\sup _{x, y \in \mathbb{X}, x \neq y} \frac{|f(x)-f(y)|}{|x-y|},$$ 
   $\text {the Lipschitz constant of } f \in \mathcal{C}\left(\mathbb{X}\right)$. Let $\mathcal{BL}\left(\mathbb{X}\right) $ be the space of bounded Lipschitz continuous functions and $\mathcal{BL}_1\left(\mathbb{X}\right) =\left\{f \in \mathcal{B} \mathcal{L}\left(\mathbb{X}\right): \mathcal{B} \mathcal{L}(f):=\|f\|_{\infty}+L_f \leq 1\right\} $.
	
	Let $\eta^x, \nu^x \in \mathcal{M}_+(I)$, then we can define the bounded Lipschitz distance as:
	$$
	d_{BL}(\eta^x, \nu^x) := \sup_{f \in \mathcal{BL}_1(\mathbb{X})} \int_\mathbb{X} f(y) ~\left( \eta^x(\txtd y) - \nu^x(\txtd y)\right).
	$$
	Hence, given $\eta, \nu \in \mathcal{B}(I, \mathcal{M}_+(I))$, define the uniform bounded Lipschitz metric:
	$$
	d_\infty(\eta, \nu) := \sup_{x \in \mathbb{X}} d_{BL}(\eta^x, \nu^x).
	$$
 and the bounded Lipschitz norm (on the space of all finite signed Borel measures):
 $$\|\eta\|:= \sup_{x \in \mathbb{X}} \sup_{f \in \mathcal{BL}_1(\mathbb{X})} \int_\mathbb{X} f(y) ~ \eta^x(\txtd y).$$
	We have the following Proposition~\cite[Proposition 2.6]{Kuehn}:
	
	\begin{prop}
		Let $I$ be a complete separable metric space. Assume $I$ is compact. Then $(\mathcal{B}(I, \mathcal{M}_+(I)), d_\infty)$ and $(\mathcal{C}(I, \mathcal{M}_+(I)), d_\infty)$ are complete metric spaces.
	\end{prop}
	
	In our problem, we will consider $I = [0, 1]$. Since we are working with DGMs, and in our problem we have a graph represented by a matrix, we need to approximate this graph by a DGM. Consider a partition of the interval $I=[0,1]$ as given by $I_i^N=]\frac{i-1}{N},\frac{i}{N}]$, for $1<i\leq N$, and $I_1^N=[0,\frac{1}{N}]$. Let us build the diagraph measures associated with the matrices $A^N_{i,j}$ and $\hat{A}^N_{i,j}$ as follows
	\begin{equation}
		\begin{split}
			\eta^u_{A^N}(v):=\sum_{i=1}^N\textbf{1}_{I_i^N}(u)\sum_{j=1}^N\frac{A^N_{i,j}}{N}\delta_{\frac{j}{N}}(v),\\
			\eta^u_{\hat{A}^N}(v):=\sum_{i=1}^N\textbf{1}_{I_i^N}(u)\sum_{j=1}^N\frac{\hat{A}^N_{i,j}}{N}\delta_{\frac{j}{N}}(v),
		\end{split}
		\label{dgmapprox}
	\end{equation}
where $\textbf{1}_{\frac{j}{N}}(x)$ is the indicator function and ${\frac{j}{N}}$ serves as the representative of the set $I_j^N$ and  . We recall that the sequence of graphs $\{A^N\}_{N\geq1}$ converges to the diagraph measure $\eta$ if and only if $d_\infty(\eta_{A^N},\eta)\rightarrow 0$, as $N\rightarrow\infty$. Similarly, we aim for $\hat{A}^N$ to converge to the diagraph measure $\hat{\eta}$. Observe that $\eta_{A^N}\in \mathcal{B}\left(I, \mathcal{M}_{+}(I)\right)\bigcap\mathcal{C}\left(I, \mathcal{M}_{+}(I)\right)$.
	
In order to analyze the collection of probability laws, let us consider the following space of probability measures
	$$\mathcal{N}:=\left\lbrace \nu\in \mathcal{P}(I\times\mathcal{C}([0,T],\mathbb{X})): \pi_1 \circ \nu=\lambda \right\rbrace,$$ 
	where $\lambda$ denotes the Lebesgue measure on $I$ and $\pi_1$ is the projection map associated to the first coordinate.
	To compare probability measures we make use of the following Wasserstein-2 metric:
	$$W_2(\mu,\nu):=\left( \inf\left\lbrace \mathbb{E}|X-\tilde{X}|^2: \mathcal{L}(X)=\mu, \mathcal{L}(\tilde{X})=\nu\right\rbrace\right) ^{\frac{1}{2}},\;\mu,\nu\in \mathcal{P}(\mathbb{R}^d).$$
	As the probability measures also depend on time, we can define the following Wasserstein-2 metric:
 \begin{align*}
W_{2,t}(\mu,\nu) := & \left( \inf \left\lbrace \mathbb{E}\|X-\tilde{X}\|_{*,t}^2 : \mathcal{L}(X) = \mu, \mathcal{L}(\tilde{X}) = \nu \right\rbrace \right)^{\frac{1}{2}}, \\
& t \in [0,T], \quad \mu, \nu \in \mathcal{P}(\mathcal{C}([0,T],\mathbb{X})),
\end{align*}
	where $\|X\|_{*,t}:=\sup_{0\leq s\leq t} |X_s|$ for $X\in\mathcal{C}([0,T],\mathbb{X})$, $t\in[0,T]$. Additionally, as we are also working with diagraph measures, we must also measure distances taking into account the graph's heterogeneity variable. For this purpose, we have the following two measures:
		$$W^{\mathcal{N},2}_{2,t}(\mu,\nu):=\left(\int_I \left[ W_{2,t}(\mu_u,\nu_u) \right] ^2~\txtd u\right)^{\frac{1}{2}}, \;t\in[0,T],\;\mu,\nu\in \mathcal{N},$$
		$$W^{\mathcal{N},\infty}_{2,t}(\mu,\nu):=\sup_{u\in I} W_{2,t}(\mu_u,\nu_u) , \;t\in[0,T],\;\mu,\nu\in \mathcal{N}.$$
		Since $I=[0,1]$, the following inequality holds:
		\[ W^{\mathcal{N},2}_{2,t}(\mu,\nu)\leq W^{\mathcal{N},\infty}_{2,t}(\mu,\nu).
		\]
	We will work in the space $\mathcal{N}$ which, equipped with the distance $W^{\mathcal{N},2}_{2,t}(\mu,\nu)$, is a complete metric space.	
	
	\begin{remark}
		Notice that the Wasserstein metric can be compared with the bounded Lipschitz metric, denoted as $d_{\mathcal{BL}}$, because the following inequality holds:
		$$W_{2,t}(\mu_u,\nu_u)\geq\sup_{f\in\mathcal{BL}_1}\left| \int_{\mathbb{R}^d}f(x)  ~\mu_{u,t}(\txtd x)-\int_{\mathbb{R}^d}f(x)  ~\nu_{u,t}(\txtd x)\right|  ,\;\mu,\nu\in \mathcal{N},$$
		i.e. $d_{\mathcal{BL}}(\mu_{u,t},\nu_{u,t})\leq W_{2,t}(\mu_u,\nu_u)$ (see \cite{Gibbs}).	
		\label{warsineq}
	\end{remark}

\section{Main Results}

Before stating our results, we introduce our main assumptions.

\paragraph{Assumptions ($H$)} 
Let us assume that $f$, $h$ and $g$ are bounded and Lipschitz continuous functions. We denote $B_f = \sup_{x \in \mathbb{X}} |f(x)|$, $B_h = \sup_{(x, y) \in \mathbb{X}^2} |h(x,y)|$ and $B_g = \sup_{(x, y) \in \mathbb{X}^2} |g(x, y)|$, and $|f(x) - f(y)| \leq L_f |x - y|$ holds for $x, y \in \mathbb{X}$, $|h(x_1, y_1) - h(x_2, y_2)| \leq L_h (|x_1 - x_2| + |y_1 - y_2|)$ and $|g(x_1, y_1) - g(x_2, y_2)| \leq L_g (|x_1 - x_2| + |y_1 - y_2|)$ for $(x_i, y_i) \in \mathbb{X}^2$, where $L_f$, $L_h$ and $L_g$ are the respective Lipschitz constants.

We remark that the Lipschitz assumptions are in many systems not a severe restriction, e.g., if the system is globally dissipative, one may use cut-off arguments to make sufficiently smooth nonlinearities Lipschitz in all parts of phase space relevant for the long-term dynamics. Of course, there are situations, where Lipschitz assumptions fail, e.g., for singular interaction terms but we are not going to pursue this direction here.

Let $[0, T]$ be the time domain for the system dynamics for some $T > 0$. Consider a filtered probability space $(\Omega, \mathcal{F}, \{F_t\}_{t\in[0,T]}, \mathbb{P})$, where $\{F_t\}$ denotes a filtration that complies with standard conditions. The particle trajectories $X(t)$ take values in $\mathbb{X}$. The random variable $U$ is given on a compact Polish probability space $(I,\mathfrak{B}(I),\mu_I)$, which we select as $I=[0,1]$ with the usual Lebesgue measure here but point out that more general choices might turn out to be useful as well. Indeed, recall that $I$ will be space to track the node labels upon passing to the mean-field limit, so choosing a more geometric space for $(I,\mathfrak{B}(I),\mu_I)$ could track distances on the graph but here we work measure-theoretically so the Borel-Maraham theorem guarantees that no important measure-theoretic information is lost by considering $I=[0,1]$. Regarding the assumptions about the graph, we assume that $\sup_{1\leq i\leq N}\sum_{j=1}^N(A^N_{i,j})^2=\mathcal{O}(N)$ and $\sup_{1\leq i\leq N}\sum_{j=1}^N(\hat{A}^N_{i,j})^2=\mathcal{O}(N)$ as $N\rightarrow \infty$. 

Finally, we need to impose conditions on the initial data. We assume that the initial data ${X_i(0)}$, for $1\leq i\leq N$, are independent and identically distributed, and we have finite second moments initially $\mathbb{E}|X_i(0)^2|<+\infty$, for all $i\in\{1,2,\ldots,N\}$. Moreover, assume that 
$$\lim_{N\rightarrow \infty}\frac{1}{N}\sum_{i=1}^{N}\mathbb{E}|X_i^{0}-X_{\frac{i}{N}}^0|^2= 0,$$ 
for any random variable $X_i^{0}$ and $X^0_{\frac{i}{N}}$, for $i=1,\ldots,N$, where $\mu^{0}=\mathcal{L}(X_i^{0})$ and $\bar{\mu}^{0}=\mathcal{L}(X_{\frac{i}{N}}^0)$, for $i=1,\ldots,N$. 

After establishing the key assumptions of our problem \eqref{eq1}, we proceed to present our main theorem, where we establish the convergence in law of the empirical measure towards a probability measure that satisfies the Vlasov-Fokker-Planck equation. The proof of this theorem will be provided in Section \ref{secTheo1}.

\begin{theorem}
	Under the assumptions ($H$), let us consider sequences of graphs $\{A_{i,j}^N\}_{N\geq1}$ and $\{\hat{A}_{i,j}^N\}_{N\geq1}$. Assume that there exists $\eta$ and $\hat{\eta}$ (DGM) such that $\eta_{A^N}$ converges to $\eta$ and $\eta_{\hat{A}^N}$ converges to $\hat{\eta}$. In this case, the empirical measure $\mu_N$ converges in probability to a measure $\bar{\mu} \in \mathcal{P}(C([0,T],\mathbb{X}))$, where $\bar{\mu}$ represents the law of the solution $X$ of equation \eqref{indep}. And  $\bar{\mu}$ solves the Vlasov-Fokker-Plank equation in the weak sense
	\begin{equation}
		\begin{split}
			\partial_t \bar{\mu}_{t} &+ \partial_x \left(\bar{\mu}_{t} f(x) + \int_I{\mu}_{u,t} \int_I \int_{\mathbb{X}} g(x,y)  ~\mu_{v,t}(\txtd y)  \eta^u(\txtd v)\txtd u\right)=\\
			&+\frac{1}{2}\partial^2_x\left( \int_I {\mu}_{u,t} \left[ \int_I \int_{\mathbb{X}} h(x,y)  ~\mu_{v,t}(\txtd y)  \hat{\eta}^u(\txtd v)\right] ^2~\txtd u\right). 
		\end{split}
		\label{FP}
	\end{equation}
	\label{teorema}
\end{theorem}

Note that when studying this problem, we have assumed that the heterogeneity variable of the graph is a random variable given by a uniform distribution. We can modify this hypothesis and obtain better properties regarding the law $\mu_{u,t}$.

As we did before, let us consider the following system of independent processes:
\begin{equation}
	\begin{split}
		X_u(t) =& X_u(0) + \int_{0}^{t} f(X_u(s)) ~\txtd s + \int_{0}^t \int_I \int_{\mathbb{X}} g(X_u(s), y)  ~\mu_{v, s}(\txtd y) \eta_u(\txtd v)\txtd s \\
		&+ \int_{0}^t \int_I \int_{\mathbb{X}} h(X_u(s), y)~\mu_{v, s}(\txtd y)\hat{\eta}^u(\txtd v) \txtd B_s^u,
	\end{split}
	\label{indep2}
\end{equation}
where $\mu_{u, t} = \mathcal{L}(X_u(t))$, for $u\in I$. Since we do not take $u$ as a random variable, we need to modify the way we understand the law of $X_u(t)$ for each $u$ and impose more properties on this law. To do this, we will modify our space, and instead of working in $\mathcal{N}$, we will work in the following space:
\begin{align*}
	\tilde{\mathcal{N}}:=\left\lbrace \nu=(\mu_u,u\in I)\in (\mathcal{P}(\mathcal{C}([0,T],\mathbb{X})))^I: u\rightarrow\mu_u \text{ is measurable},\right. \\
	\left.\sup _{u \in I} \int_{\mathcal{C}_d}\|x\|_{*, T}^2 ~\mu_u(\txtd  x)<\infty\right\rbrace.
\end{align*}
This space equipped with the distance $W^{\mathcal{N},\infty}_{2,t}$ is a complete metric space.	

In addition to the assumptions ($H$), we must establish one more set of assumptions:

\paragraph{Assumptions ($\tilde{H}$)} Let us assume that the law of probability $\bar{\mu}_0$ of the initial condition is measurable with respect to the heterogeneity variable of the graph. In other words, the mapping $u\in I\rightarrow\bar{\mu}^0_u\in P(\mathbb{X})$ is measurable. Moreover, assume that $W_2\left(\bar{\mu}_{u_1}^0, \bar{\mu}_{u_2}^0\right) \leq \alpha\left|u_1-u_2\right|$, for $u_1,u_2\in I$ and $\alpha\in\mathbb{R}_+$.

With the additional assumptions about our system, we obtain the following result:

\begin{theorem}
Under the assumptions (H) and ($\tilde{H}$), let us consider a sequence of graphs $\{A_{i,j}^N\}_{N\geq1}$ and $\{\hat{A}_{i,j}^N\}_{N\geq1}$. Assume that there exists $\eta$ and $\hat{\eta}$ (DGM) such that $\eta_{A^N}$ converges to $\eta$ and $\eta_{\hat{A}^N}$ converges to $\hat{\eta}$. In this case, the empirical measure $\mu_N$ converges in probability  in $\mathcal{P}(C([0,T],\mathbb{X}))$ to the measure $\bar{\mu}$. And  $\bar{\mu}$ solves the Vlasov-Fokker-Plank equation in the weak sense
\begin{equation}
	\begin{split}
		\partial_t \bar{\mu}_{t} &+ \partial_x \left(\bar{\mu}_{t} f(x) + \int_I{\mu}_{u,t} \int_I \int_{\mathbb{X}} g(x,y) ~ \mu_{v,t}( \txtd y) \eta^u( \txtd v)\txtd u\right)=\\
		&+\frac{1}{2}\partial^2_x\left( \int_I {\mu}_{u,t} \left[ \int_I \int_{\mathbb{X}} h(x,y) ~\mu_{v,t}( \txtd y) \hat{\eta}^u( \txtd v)\right] ^2 ~\txtd u\right). 
	\end{split}
	\label{FP2}
\end{equation}
\label{teorema2}
\end{theorem}

Observing that in both theorems, we obtain similar results regarding the probability $\bar{\mu}$ that is, in both, we demonstrate that our empirical measure converges to it in probability and satisfies the Vlasov-Fokker-Planck equation. However, in each case, the law represents a different concept.

In the first case, $\mu_{u,t}(x)$ belongs to the space $\mathcal{P}(I\times C([0,t],\mathbb{X}))$, so we are considering that our probability measure determines the probability of finding a particle at position $x$ with heterogeneity variable $u$, for each time instant $t\in [0,T]$. In contrast, in the second case, for each fixed $u\in I$, the probability measure $\mu_{u,t}(x)$ determines the probability of finding the particle at position $x$.

These different ways of understanding each probability measure lead us to observe $\bar{\mu}$ in different ways: In the first case, it is understood as a marginal as we integrate over all possible values that $u$ can take, giving us $\bar{\mu}$ as the probability law of $X$. Meanwhile, in the second case, since we have a probability law for each $u$, by integrating with respect to $u$, we are measuring the average probability of finding particles at position $X$.

	\section{Existence of Solutions}\label{secIndep}
	First of all, we are going to prove the existence and uniqueness of solution of a solution to \eqref{indep}.
 
	\begin{prop}\label{propindep}
		Under the assumptions ($H$), for every random variable U on I there exists a unique solution to \eqref{indep}.
	\end{prop}
\begin{proof}
	
	To prove the existence of solution for the equation \eqref{indep}, we will consider an operator defined on the space $\mathcal{N}$, and look for a fixed point of this operator. That is, let us consider the mapping $\mu\in\mathcal{N}\mapsto\mathcal{F}(\mu)\in\mathcal{N}$, where $\mathcal{F}(\mu)$ is the law associated to the solution of the equation
	\begin{equation}
			\begin{split}
	X^\mu_u(t)=&X^\mu_u(0)+\int_{0}^{t}f(X^\mu_u(s))~\txtd s+\int_{0}^t\int_I\int_{\mathbb{X}} g(X^\mu_u(s),y)~\mu_{v,s}(\txtd y)\eta^u(\txtd v) \txtd s\\
			&+ \int_{0}^t\int_I\int_{\mathbb{X}} h(X^\mu_u(s),y)~\mu_{v,s}( \txtd y)\hat{\eta}^u(\txtd v) \txtd B_s^u.
		\end{split}
		\label{eq2}
	\end{equation}
	Note that if we have a fixed point, then $\mathcal{F}(\mu)=\mathcal{L}(X^\mu)=\mu$, so we would prove the existence of solution. First of all, we must verify that the operator is well defined and that for each $\mu\in\mathcal{N}$ we have existence of solution of the problem. Let us first prove that the operator is well defined. To do this, let us take $\mu\in\mathcal{N}$, and  $X^0_u(t)=X_u(0)$, $\forall t\in[0,T]$ and $u\in I$. Consider the following recurrence equation
	\begin{eqnarray}
			X^n_u(t)\hspace{-0.25cm}&=\hspace{-0.25cm}&X^{n-1}_u(0)+\int_{0}^{t}f(X^{n-1}_u(s)) \txtd s+\int_{0}^t\int_I\int_{\mathbb{X}} g(X^{n-1}_u(s),y)~\mu_{v,s}( \txtd y)\eta^u( \txtd v) \txtd s\nonumber\\
		&&+\int_{0}^t\int_I\int_{\mathbb{X}} h(X^{n-1}_u(s),y)~\mu_{v,s}(\txtd y)\hat{\eta}^u(\txtd v)\txtd B_s^u,
		\label{induction}
	\end{eqnarray}
	where  $X^k_u(0)=X^0_u(0)$, for all $k\geq 1$.
	Let us prove that $\left\lbrace X^n_u\right\rbrace _{n\geq0}$ is Cauchy. 
	Since $X^{n+1}_u - X^n_u$ is a martingale, we can use Burkholder-Davis-Gundy's inequality ( \cite{Evans}) and obtain
	\begin{equation}
		\mathbb{E}\|X^{n+1}_u-X^n_u\|^2_{*,t}\leq K_{BDG}\mathbb{E}[X^{n+1}_u-X^n_u]_t.
		\label{cauchyineq}
	\end{equation}
	In this case the $K_{BDG}=4$. Estimating the right-hand term we have
	\begin{align*}
		&\mathbb{E}|X^{n+1}_u(t)-X^n_u(t)|^2 \leq 3\mathbb{E}\int_{0}^{t}|f(X^{n}_u(s))-f(X^{n-1}_u(s))|^2 ~\txtd s \\
		& +3\mathbb{E}\int_{0}^{t}\left|\int_I\int_{\mathbb{X}} g(X^{n}_u(s),y)\mu_{v,s}( \txtd y)\eta^u( \txtd v)-g(X^{n-1}_u(s),y)\mu_{v,s}( \txtd y)\eta^u( \txtd v)\right|^2 \txtd s \\
		& +3\mathbb{E}\left|\int_{0}^t\int_I\int_{\mathbb{X}}\left( h(X^{n}_u(s),y)\mu_{v,s}(\txtd y)\hat{\eta}^u(\txtd v)-h(X^{n-1}_u(s),y)\mu_{v,s}( \txtd y)\hat{\eta}^u(\txtd v)\right) \txtd B_s^u\right|^2.
	\end{align*}
	 Using properties of stochastic integrals \cite{Peres}, i.e., by It\^o isometry, we can rewrite the last integral term to convert it to a deterministic integral and obtain
	 \begin{align*}
	 	\mathbb{E}&|X^{n+1}_u(t)-X^n_u(t)|^2\leq3\mathbb{E}\int_{0}^{t}|f(X^{n}_u(s))-f(X^{n-1}_u(s))|^2 ~\txtd s\\
	 	&+3\mathbb{E}\int_{0}^{t}\left|\int_I\int_{\mathbb{X}} g(X^{n}_u(s),y)~\mu_{v,s}(\txtd y)\eta^u(\txtd v)-g(X^{n-1}_u(s),y)~\mu_{v,s}(\txtd y)\eta^u(\txtd v)\right|^2 \txtd s\\
	 	&+3\mathbb{E}\int_{0}^t\left|\int_I\int_{\mathbb{X}} h(X^{n}_u(s),y)~\mu_{v,s}(\txtd y)\hat{\eta}^u(\txtd v)-h(X^{n-1}_u(s),y)~\mu_{v,s}(\txtd y)\hat{\eta}^u(\txtd v)\right|^2\txtd s.
	 \end{align*}
	Using that the functions $g$, $h$ and $f$ are Lipschitz, we get
	\begin{align*}
		\mathbb{E}|X^{n+1}_u(t)-X^n_u(t)|^2&\leq3\mathbb{E}\int_{0}^{t}L_f^2|X^{n}_u(s)-X^{n-1}_u(s)|^2~\txtd s\\
		&+3\mathbb{E}\int_{0}^{t}\left(\int_I\int_{\mathbb{X}} L_g|X^{n-1}_u(s)-X^n_u(s)|~\mu_{v,s}(\txtd y)\eta^u(\txtd v)\right)^2\txtd s\\
		&+3\mathbb{E}\int_{0}^{t}\left(\int_I\int_{\mathbb{X}} L_h|X^{n-1}_u(s)-X^n_u(s)|~\mu_{v,s}(\txtd y)\hat{\eta}^u(\txtd v)\right)^2\txtd s	.	
	\end{align*}
	Therefore, we get
 $$\mathbb{E}|X^{n+1}_u(t)-X^n_u(t)|^2\leq3\left(L_f^2+L_g^2\|\eta\|^2+L_h^2\|\hat{\eta}\|^2\right)\mathbb{E}\int_{0}^{t}|X^{n}_u(s)-X^{n-1}_u(s)|^2~\txtd s.$$
	By taking the supremum in $s$ of the difference between $X^{n}$ and $X^{n-1}$, in the right-hand side integral, it follows
	$$\mathbb{E}|X^{n+1}_u(t)-X^n_u(t)|^2\leq3\left(L_f^2+L_g^2\|\eta\|^2+L_h^2\|\hat{\eta}\|^2\right)\mathbb{E}\int_{0}^{t}\|X^{n}_u-X^{n-1}_u\|_{*,s}^2~\txtd s.$$
	By using \eqref{cauchyineq}, we arrive at the expression:
	$$\mathbb{E}\|X^{n+1}_u-X^n_u\|^2_{*,t}\leq 12\left(L_f^2+L_g^2\|\eta\|^2+L_h^2\|\hat{\eta}\|^2\right)\int_{0}^{t}\mathbb{E}\|X^{n}_u-X^{n-1}_u\|_{*,s}^2~\txtd s.$$
	Defining $C=12\left(L_f^2+L_g^2\|\eta\|^2+L_h^2\|\hat{\eta}\|^2\right)$ and iterating this, we get that:
	$$\mathbb{E}\|X^{n+1}_u-X^n_u\|^2_{*,t}\leq C^n\frac{T^n}{n!}\mathbb{E}\|X^{1}_u-X^{0}_u\|_{*,s}^2,$$
	where $\mathbb{E}\|X^{1}_u-X^{0}_u\|_{*,s}^2$ is bounded, due to the assumptions on the initial data and the fact that the functions $f$,$g$ and $h$ are bounded. It follows that the sequence $\left\lbrace X^n_u\right\rbrace _{n\geq0}$ is Cauchy and converges uniformly in probability at $u\in I$ to $X^\mu_u$, which satisfies the equation \eqref{eq2}. Therefore, we have that $\mathcal{F}$ is a well-defined map from $\mathcal{N}$ to $\mathcal{N}$.
	It remains to be seen that the application $\mathcal{F}$ has a fixed point in this space. To do this, let us consider $\mu,\nu\in\mathcal{N}$, with $X^\mu_u$ and $X^\nu_u$ their respective solutions to \eqref{eq2}, and let us estimate $\mathbb{E}|X^\mu_u-X^\nu_u|^2$ as follows
	\begin{align*}
		\mathbb{E}|X^\mu_u-X^\nu_u|^2&\leq3\mathbb{E}\int_0^t|f(X^\mu_u(s))-f(X^\nu_u(s))|^2~\txtd s\\
		&+3\mathbb{E}\int_{0}^{t}\int_I\left|\int_{\mathbb{X}} g(X^\mu_u(s),y)\mu_{v,s}(\txtd y)-g(X^\nu_u(s),y)\nu_{v,s}(\txtd y)\right|^2\eta^u(\txtd v)\txtd s\\
		&+3\mathbb{E}\int_{0}^t\int_I\left|\int_{\mathbb{X}} h(X^{\mu}_u(s),y)\mu_{v,s}(\txtd y)-h(X^{\nu}_u(s),y)\nu_{v,s}(\txtd y)\right|^2\hat{\eta}^u(\txtd v)\txtd s,
	\end{align*}
	where we have used again It\^o's isometry.
	By adding and subtracting $g(X^\mu_u(s),y)\nu_{v,s}(\txtd y)$ in the second term, and $h(X^\mu_u(s),y)\nu_{v,s}(\txtd y)$ we can bound the above expression by the sum of the following terms.
\begin{align}
	\mathbb{E}|X^\mu_u-X^\nu_u|^2 &\leq 3\mathbb{E}\int_0^t|f(X^\mu_u(s))-f(X^\nu_u(s))|^2~\txtd s \nonumber \\
	&+ 6\mathbb{E}\int_{0}^{t}\int_I\left|\int_{\mathbb{X}} g(X^\mu_u(s),y)~(\mu_{v,s}(\txtd y)-\nu_{v,s}(\txtd y))\right|^2\eta^u(\txtd v)\txtd s \nonumber \\
	&+ 6\mathbb{E}\int_{0}^{t}\int_I\left|\int_{\mathbb{X}} (g(X^\mu_u(s),y)-g(X^\nu_u(s),y))~\nu_{v,s}(\txtd y)\right|^2\eta^u(\txtd v)\txtd s \nonumber\\
	&+ 6\mathbb{E}\int_{0}^{t}\int_I\left|\int_{\mathbb{X}} h(X^\mu_u(s),y)(\mu_{v,s}(\txtd y)-\nu_{v,s}(\txtd y))\right|^2\hat{\eta}^u(\txtd v)\txtd s \nonumber \\
	&+ 6\mathbb{E}\int_{0}^{t}\int_I\left|\int_{\mathbb{X}} (h(X^\mu_u(s),y)-h(X^\nu_u(s),y))\nu_{v,s}(\txtd y)\right|^2\hat{\eta}^u(\txtd v)\txtd s \nonumber.
\end{align}
	Using the Lipschitz condition of the functions $f$, $g$ and $h$, and the Remark \ref{warsineq}, we have
	\begin{align}
		\mathbb{E}|X^\mu_u-X^\nu_u|^2&\leq (3L_f^2+6(L_g^2\|\eta\|+L_h^2\|\hat{\eta}\|))\mathbb{E}\int_{0}^{t} |X^\mu_u(s)-X^\nu_u(s)|^2~\txtd s\nonumber\\
		&+6B_g^2\int_{0}^{t}\int_I |W_{2,s}(\mu_{v,s},\nu_{v,s})|^2~\eta^u(\txtd v)\txtd s\nonumber\\
		&+6B_h^2\int_{0}^{t}\int_I |W_{2,s}(\mu_{v,s},\nu_{v,s})|^2~\hat{\eta}^u(\txtd v)\txtd s.\nonumber
	\end{align}
	Taking the supremum for $v\in I$ in the 2-Wasserstein metric, we arrive at the following expression:
	\begin{align*}
		\mathbb{E}|X^\mu_u-X^\nu_u|^2&\leq (3L_f^2+6(L_g^2\|\eta\|+L_h^2\|\hat{\eta}\|))\mathbb{E}\int_{0}^{t} |X^\mu_u(s)-X^\nu_u(s)|^2~\txtd s\\
		&+6(B_g^2\|\eta\|+B_h^2\|\hat{\eta}\|)\int_{0}^{t} |W^{\mathcal{N},\infty}_{2,s}(\mu_{s},\nu_{s})|^2~\txtd s.
	\end{align*}
		It then, follows from Gr\"onwall's inequality that
		$$\mathbb{E}|X^\mu_u-X^\nu_u|^2\leq 6(B_g^2\|\eta\|+B_h^2\|\hat{\eta}\|)e^{(3L_f^2+6(L_g^2\|\eta\|+L_h^2\|\hat{\eta}\|))T}\int_{0}^{t} |W^\mathcal{N,\infty}_{2,s}(\mu,\nu)|^2~\txtd s.$$
		By this inequality and using the Burkholder-Davis-Gundy (BDG) inequality again, we get that
		$$\mathbb{E}\|X^\mu_u-X^\nu_u\|_{*,t}^2\leq 24(B_g^2\|\eta\|+B_h^2\|\hat{\eta}\|)e^{(3L_f^2+6(L_g^2\|\eta\|+L_h^2\|\hat{\eta}\|))T}\int_{0}^{t} |W^\mathcal{N,\infty}_{2,s}(\mu,\nu)|^2~\txtd s.$$
		This means that
		\begin{equation}
			|W^{\mathcal{N},\infty}_{2,t}(\mathcal{F}(\mu),\mathcal{F}(\nu))|^2\leq C\int_{0}^{t} |W^{\mathcal{N},\infty}_{2,s}(\mu,\nu)|^2~\txtd s,
			\label{desuninon}
		\end{equation}
		where $C=24(B_g^2\|\eta\|+B_h^2\|\hat{\eta}\|)e^{(3L_f^2+6(L_g^2\|\eta\|+L_h^2\|\hat{\eta}\|))T}$. This expression gives the path-wise uniqueness of the solution, and also allow us to prove the existence of solution. For this purpose we will build an iterative process, as follows. Consider $\nu=(\mathcal{L}(Z_u),u\in I)$, where $Z_u(t)=X_u(0)$ for all $u\in I $ and $t\in[0,T]$. Iterating this, and using \eqref{desuninon}, we get
		$$W^{\mathcal{N},2}_{2,T}(\mathcal{F}^{n+1}(\nu),\mathcal{F}^n(\nu))\leq \frac{C^nT^n}{n!} |W^{\mathcal{N},\infty}_{2,T}(\mathcal{F}(\nu),\nu)|.$$
		It follows that the sequence $\mathcal{F}^n(\nu)$ is Cauchy for $n$ large enough, where we have used that $W^{\mathcal{N},\infty}_{2,T}(\mathcal{F}(\nu),\nu)<\infty$, due to the assumptions on the initial data and the fact that the functions $f$, $g$, and $h$ are bounded. This sequence will have a limit since $\mathcal{N}$ is a complete metric space, and hence there exists $\mu=(\mathcal{L}(X_u),u\in I)\in \mathcal{N}$ solution of \eqref{indep}, for the initial data $X_u(0)$.
	\end{proof}
	
		\begin{lemma}
			Let $\bar{\mu}$ be the law associated to the solution $X$ of \eqref{indep}, with $\mu_{u, t} = \mathcal{L}(X(s)|U = u)$ being the conditional probability with respect to the random variable $U$. Then $\mu_{u, t}$ is a weak solution to the following Vlasov-Fokker-Planck equation:
			\begin{equation}
				\begin{split}
					\partial_t {\mu}_{u,t} &+ \partial_x \left({\mu}_{u,t} f(x) + {\mu}_{u,t} \int_I \int_{\mathbb{X}} g(x,y) ~\mu_{v,t}(\txtd y) \eta^u(\txtd v)\right)=\\
					&+\frac{1}{2}\partial^2_x\left(  {\mu}_{u,t} \left[ \int_I \int_{\mathbb{X}} h(x,y) ~\mu_{v,t}(\txtd y) \hat{\eta}^u(\txtd v)\right] ^2\right) .
				\end{split}
				\label{ito1}
			\end{equation}
		\end{lemma}
		
		\begin{proof}
			Consider a test function $\phi \in C^\infty_{\textnormal{c}}(\mathbb{X})$, and let us compute the time derivative of $\phi(X_u(t))$, where $u$ is a sample of the random variable $U$. To do that, let us make use of It\^o's formula. We have
			\begin{equation}
				\begin{split}
					\txtd \phi =& \left( \partial_x \phi f(X_u(t)) + \partial_x \phi \int_I \int_{\mathbb{X}} g(X_u(t),y) ~\mu_{v,t}(\txtd y) \eta^u(\txtd v) \right. \\
					&\left. + \frac{1}{2}\left[ \int_I \int_{\mathbb{X}} h(x,y) \mu_{v,t}(\txtd y) \hat{\eta}^u(\txtd v)\right] ^2 \partial^2_x \phi \right) \txtd t \\
					&+ \partial_x \phi\int_I \int_{\mathbb{X}} h(x,y) \mu_{v,t}(\txtd y) \hat{\eta}^u(\txtd v) \txtd B^u_t.
				\end{split}
			\end{equation}
			
			If we calculate the expectation of this expression and integrate by parts, we obtain the weak formulation \eqref{ito1}.
		\end{proof}
		
		With the objective of demonstrating the convergence of the solution, it is crucial to investigate, how the solution depends on the digraph measure $\eta$. Specifically, we want to establish suitable continuity properties. Such continuity proofs will lay the foundation for later proving the convergence to a mean-field limit.

		\begin{prop}\label{propdgm}
			Given $\hat{\eta}_1, \hat{\eta}_2 \in \mathcal{BC}(I, \mathcal{M}_+(I))$, let $\mu_1$ and $\mu_2$ be the laws of the solutions of \eqref{indep} for the DGMs $\hat{\eta}_1$ and $\hat{\eta}_2$. Then,
			\begin{equation}
				\left[ W_{2,s}^{\mathcal{N},2}(\mu^1,\mu^2)\right]^2\leq \hat{C}_1\hat{C}_2\left[ d_\infty(\hat{\eta}_1,\hat{\eta}_2)\right] ^2 e^{\hat{C}_1T},
			\end{equation}
			where $\hat{C}_1 =3\left[ 2B_g^2\|\eta\|^2+3B_h^2\|\hat{\eta}_1\|^2\right]e^{\left(3L_f^2+3\left[ 2L_g^2\|\eta\|^2+3L_h^2\|\hat{\eta}_1\|^2\right] \right) T}$ and \newline $\hat{C}_2=\frac{3B^2_hT}{\left[ 2B_g^2\|\eta\|^2+3B_h^2\|\hat{\eta}_1\|^2\right]}$.

			Given $\eta_1, \eta_2 \in \mathcal{BC}(I, \mathcal{M}_+(I))$, let $\mu_1$ and $\mu_2$ be the laws of the solutions of \eqref{indep} for the DGMs $\eta_1$ and $\eta_2$. Then,
			\begin{equation}
				\left[ W_{2,s}^{\mathcal{N},2}(\mu^1,\mu^2)\right]^2\leq C_1C_2\left[ d_\infty(\eta_1,\eta_2)\right] ^2 e^{C_1T},
			\end{equation}
			where $C_1 =3\left[ 3B_g^2\|\eta_1\|^2+2B_h^2\|\hat{\eta}\|^2\right]e^{\left(3L_f^2+3\left[ 3L_g^2\|\eta_1\|^2+2L_h^2\|\hat{\eta}\|^2\right] \right) T}$ and \newline $C_2=\frac{3B^2_gT}{\left[ 3B_g^2\|\eta_1\|^2+2B_h^2\|\hat{\eta}\|^2\right]}$.
			
		\end{prop}
		\begin{proof}
			We will proceed to prove only the first statement of the proposition, as the proof idea is the same for both cases.
			Consider two different DGMs $\hat{\eta}_1$ and $\hat{\eta}_2 \in \mathcal{B}(I, \mathcal{M}_+(I))\bigcap\mathcal{C}(I, \mathcal{M}_+(I))$ with the respective solutions they induce denoted by $X^1$ and $X^2$. We have
			\begin{align}
				&\mathbb{E}|X_u^1(t)-X_u^2(t)|^2  \nonumber \\
				&\leq 3\mathbb{E}\int_{0}^t|f(X_u^1(s))-f(X_u^2(s))|^2 ~ \txtd s \nonumber \\
				&+3\mathbb{E}\int_0^t\left|\int_I\int_{\mathbb{X}}g(X^1_u(s),y)~\mu_{v,s}^1(\txtd y)\eta^u(\txtd v)-g(X^2_u(s),y)~\mu_{v,s}^2(\txtd y)\eta^u(\txtd v)\right|^2 \, \txtd s  \nonumber\\
				&+3\mathbb{E}\left|\int_0^t\int_I\int_{\mathbb{X}}\left(h(X^1_u(s),y)~\mu_{v,s}^1(\txtd y)\hat{\eta}^u_1(\txtd v)-h(X^2_u(s),y)~\mu_{v,s}^2(\txtd y)\hat{\eta}^u_2(\txtd v)\right)\,\txtd B^u_s\right|^2  . \nonumber
			\end{align}
			Let us start with the term related to Brownian motion. We can rewrite this term again using standard properties of stochastic integrals. Furthermore, by adding and subtracting terms, we get
			\begin{align*}
				3\mathbb{E}&\left|\int_0^t\int_I\int_{\mathbb{X}}\left(h(X^1_u(s),y)~\mu_{v,s}^1(\txtd y)\hat{\eta}^u_1(\txtd v)-h(X^2_u(s),y)~\mu_{v,s}^2(\txtd y)\hat{\eta}^u_2(\txtd v)\right)\,\txtd B^u_s\right|^2  \nonumber\\
				\leq &+9\mathbb{E}\int_0^t\left|\int_I\int_{\mathbb{X}}\left[ h(X^1_u(s),y)-h(X^2_u(s),y)\right] \mu_{v,s}^1(\txtd y)\hat{\eta}^u_1(\txtd v)\right|^2 \, \txtd s  \nonumber\\
				&+9\mathbb{E}\int_0^t\left|\int_I\int_{\mathbb{X}}h(X^2_u(s),y)\left[\mu_{v,s}^1(\txtd y)-\mu_{v,s}^2(\txtd y)\right] \hat{\eta}^u_1(\txtd v)\right|^2 \, \txtd s  \nonumber\\
				&+9\mathbb{E}\int_0^t\left|\int_I\int_{\mathbb{X}}h(X^2_u(s),y)~\mu_{v,s}^2(\txtd y)\left[\hat{\eta}^u_1(\txtd v)-\hat{\eta}^u_2(\txtd v)\right] \right|^2 \, \txtd s . \nonumber
			\end{align*}
			Using the properties of $h$, and just like we did in the previous proof, we can bound these terms as follow
			\begin{align*}
                3\mathbb{E}&\left|\int_0^t\int_I\int_{\mathbb{X}}\left(h(X^1_u(s),y)\mu_{v,s}^1(\txtd y)\hat{\eta}^u_1(\txtd v)-h(X^2_u(s),y)\mu_{v,s}^2(\txtd y)\hat{\eta}^u_2(\txtd v)\right)\,\txtd B^u_s\right|^2   \nonumber\\
				\leq &+9L_h^2\|\hat{\eta}_1\|^2\mathbb{E}\int_0^t\left|X^1_u(s)-X^2_u(s)\right|^2 \, \txtd s  \nonumber\\
				&+9B_h^2\|\hat{\eta}_1\|^2\int_0^t\left|W_{2,s}^{\mathcal{N},\infty}(\mu^1,\mu^2)\right|^2 \, \txtd s  \nonumber\\
				&+9B_h^2T\left[ d_\infty(\hat{\eta}_1,\hat{\eta}_2)\right] ^2  \nonumber.
			\end{align*}
			If we go back to the original inequality and bound the remainder terms using the same methodology as we did for the Brownian term, we obtain the following expression
			\begin{align}
			\mathbb{E}|X_u^1(t)-X_u^2(t)|^2 &\leq \left(3L_f^2+3\left[ 2L_g^2\|\eta\|^2+3L_h^2\|\hat{\eta}_1\|^2\right] \right)\mathbb{E}\int_{0}^t |X_u^1-X_u^2|_{*,s}^2 \, \txtd s  \nonumber \\
				&+3\left[ 2B_g^2\|\eta\|^2+3B_h^2\|\hat{\eta}_1\|^2\right]\int_0^t\left[ W_{2,s}^{\mathcal{N},\infty}(\mu^1,\mu^2)\right] ^2 \, \txtd s \nonumber \\
				&+9B_h^2T\left[ d_\infty(\hat{\eta}_1,\hat{\eta}_2)\right] ^2 \nonumber.
			\end{align}
			Applying the BDG inequality and Gr\"onwall's inequality we get
			\begin{equation*}
			\mathbb{E}|X_u^1-X_u^2|_{*,t}^2\leq C_1\left(\int_0^t\left[ W_{2,s}^{\mathcal{N},\infty}(\mu^1,\mu^2) \right] ^2\, \txtd s+C_2\left[ d_\infty(\hat{\eta}_1,\hat{\eta}_2)\right] ^2\right), 
			\end{equation*}
			where $C_1=3\left[ 2B_g^2\|\eta\|^2+3B_h^2\|\hat{\eta}_1\|^2\right]e^{\left(3L_f^2+3\left[ 2L_g^2\|\eta\|^2+3L_h^2\|\hat{\eta}_1\|^2\right] \right) T}$ and \newline $C_2=\frac{3B^2_hT}{\left[ 2B_g^2\|\eta\|^2+3B_h^2\|\hat{\eta}_1\|^2\right]}$. 
            By definition of $W_{2,t}^{\mathcal{N},\infty}(\mu^1,\mu^2)$, we have the following:
			\begin{equation*}
				\left[ W_{2,s}^{\mathcal{N},\infty}(\mu^1,\mu^2)\right]^2\leq C_1\left(\int_0^t\left[ W_{2,s}^{\mathcal{N},\infty}(\mu^1,\mu^2)\right] ^2 \, \txtd s+C_2\left[ d_\infty(\hat{\eta}_1,\hat{\eta}_2)\right] ^2\right). 
			\end{equation*}
			Applying Gr\"onwall's inequality, we find
			\begin{equation*}
				\left[ W_{2,s}^{\mathcal{N},\infty}(\mu^1,\mu^2)\right]^2\leq C_1C_2\left[ d_\infty(\hat{\eta}_1,\hat{\eta}_2)\right] ^2 e^{C_1T}.
			\end{equation*}
			And thus, we obtain the desired inequality
			\begin{equation*}
				\left[ W_{2,s}^{\mathcal{N},2}(\mu^1,\mu^2)\right]^2\leq C_1C_2\left[ d_\infty(\hat{\eta}_1,\hat{\eta}_2)\right] ^2 e^{C_1T}.
			\end{equation*}
        This finishes the proof.
		\end{proof}
	
		\section{Proof of Theorem \ref{teorema}}\label{secTheo1}
		
		Once the existence of the solution of the independent problem has been proven and the continuity of the solution with respect to the DGM has been shown, the next step is to determine how close the solutions of the problem \eqref{eq1} are to the solutions of the independent problem \eqref{indep}. To compare trajectories of solutions between equation \eqref{eq1} and equation \eqref{indep}, we will associate the particles $X_i$ with the particles $X_{u(i)}$. To achieve this, we will consider a partition of the interval $I$ defined by $I^N_i$, as introduced previously. For each interval, we will select a representative $u(i) = \frac{i}{N}$. Consequently, we will compare the trajectories of $X_i^N$ and $X_{\frac{i}{N}}$. To proceed with this matter, it is important to consider that each one is described by a distinct stochastic process. However, to identify both particles as similar, we make the assumption that for our choice of the random variable $u=\frac{i}{N}$, both particles exhibit the same Brownian motions, i.e., $B^i_t=B^{\frac{i}{N}}_t$. This strategy of aligning the trajectories based on $u$ enables us to evaluate the similarity between the two particles, despite their different stochastic processes.
		
		\begin{prop}
		Under the hypotheses of Theorem \ref{teorema}, the following inequality
		\begin{align*}
			\frac{1}{N}\sum_{i=1}^{N}\mathbb{E}\|X_i^N(t)-&X_{\frac{i}{N}}\|^2_{*,t} \nonumber\leq C\frac{16}{N}\sum_{i=1}^{N}\mathbb{E}|X_i^{0}-X_{\frac{i}{N}}^0|^2 \\
			&+C \left( \frac{2}{N}(B_g^2+B_h^2)\mathcal{O}(1)+B_g^2\left[ d_{\infty}\left(\eta_{A^N}, \eta\right)\right] ^2+B_h^2\left[ d_{\infty}\left(\eta_{\hat{A}^N}, \hat{\eta}\right)\right] ^2 \right) \nonumber,
		\end{align*}
		holds, where $C=48Te^{(16L_f^2+96(L_g^2+L_h^2)\mathcal{O}(1))T}$.
		\label{compare}
		\end{prop}
		\begin{proof}
			Fix $t\in[0,T]$. Let us compare $X_i^N$ and $X_{\frac{i}{N}}$, solutions of \eqref{eq1} and \eqref{indep}, respectively.
			\begin{align*}
				&\frac{1}{N}\sum_{i=1}^{N}\mathbb{E}\|X_i^N(t)-X_{\frac{i}{N}}\|^2_{*,t}
				\leq\frac{16}{N}\sum_{i=1}^{N}\mathbb{E}|X_i^{0}-X_{\frac{i}{N}}^0|^2\\
				&+ \frac{16}{N}\sum_{i=1}^{N}\mathbb{E}\int_0^t|f(X_i^N(s))-f(X_{\frac{i}{N}}(s))|^2~\txtd s\\
				&+\frac{16}{N}\sum_{i=1}^N\mathbb{E}\int_0^t\left|\frac{1}{N}\sum_{i=1}^NA_{i,j}^Ng(X_i^N(s),X_j^N(s))-\int_I\int_{\mathbb{X}}g(X_{\frac{i}{N}}(s),y)~\mu_{v,s}(\txtd y)\eta^{\frac{i}{N}}(\txtd v)\right|^2\txtd s\\
				&+\frac{16}{N}\sum_{i=1}^N\mathbb{E}\int_{0}^t\left| \frac{1}{N} \sum_{j=1}^{N} \hat{A}^{N}_{ij} h(X_i^N(s),X_j^N(s))\txtd s-\int_I \int_{\mathbb{X}} h(X_{\frac{i}{N}}(s), y)~\mu_{v, s}(\txtd y)\hat{\eta}^{\frac{i}{N}}(\txtd v)\right|^2\txtd s.				
			\end{align*}
			Hence, we get that $\frac{1}{N}\sum_{i=1}^{N}\mathbb{E}\|X_i^N(t)-X_{\frac{i}{N}}\|^2_{*,t}\leq \mathrm{E}_0+\mathrm{E}_f+\mathrm{E}_g+\mathrm{E}_h$, so let us estimate each term. Notice that, by the construction of the solution, we have assumed that $B^i_t = B^{\frac{i}{N}}$. 		
			Let us start with $E_f$. We can bound the term depending on the function $f$ in a similar way to what we have been doing throughout this work. Using the fact that the function $f$ is Lipschitz,
		\begin{align*}
			\mathrm{E}_f & \leq L_f^2\frac{16}{N}\sum_{i=1}^{N}\int_0^t\mathbb{E}|X_i^N(s)-X_{\frac{i}{N}}(s)|^2~\txtd s.
		\end{align*}
		To estimate $E_g$, we will add and subtract different similar terms to find
		\[
			\begin{split}
				\mathrm{E}_g  \leq \frac{48}{N}&\Bigg(\int_{0}^t\sum_{i=1}^N\mathbb{E}\left|\frac{1}{N}\sum_{j=1}^NA^N_{i,j}(g(X_i^N(s),X_j^N(s))-g(X_{\frac{i}{N}}(s),X_{\frac{j}{N}}(s)))\right|^2\txtd s \\
				& + \int_{0}^t\sum_{i=1}^N\mathbb{E}\left|\frac{1}{N}\sum_{j=1}^NA^N_{i,j}\left( g(X_{\frac{i}{N}}(s),X_{\frac{j}{N}}(s))-\int_\mathbb{X}g(X_{\frac{i}{N}}(s),y)~\mu_{\frac{j}{N},s}(\txtd y)\right) \right|^2\txtd s  \\
				& + \int_{0}^t\sum_{i=1}^N\mathbb{E}\left|\int_I\int_\mathbb{X}g(X_{\frac{i}{N}}(s),y)~\mu_{v,s}(\txtd y)\left( \eta_{A^N}^{\frac{i}{N}}(\txtd v)-\eta^{\frac{i}{N}}(\txtd v)\right) \right|^2~\txtd s\Bigg),
			\end{split}
		\]
		where $\eta^{\frac{i}{N}}_{A^N}$ is given by \eqref{dgmapprox}. We can denote the previous estimate as $\mathrm{E}_g\leq\mathrm{E}^1_g+\mathrm{E}^2_g+\mathrm{E}^3_g$. Let us bound each term separately.
		Using the Lipschitz property of $g$, the term $\mathrm{E}^1_g$ can be bounded as follows:
		\begin{align*}
			\mathrm{E}^1_g & \leq L_g^2\frac{48}{N^3}\int_{0}^t\sum_{i=1}^N\left[  \sum_{j=1}^N(A^N_{i,j})^2 \cdot  \mathbb{E}\left[\sum_{j=1}^N\left( |X_i^N(s)-X_{\frac{i}{N}}(s)|+|X_j^N(s)-X_{\frac{j}{N}}(s)|\right) ^2\right]\right]\txtd s \nonumber \\
			& \leq L_g^2\frac{96}{N^3}\int_{0}^t\sum_{i=1}^N\left[  \sum_{j=1}^N(A^N_{i,j})^2 \cdot\sum_{j=1}^N\left(  \mathbb{E}|X_i^N(s)-X_{\frac{i}{N}}(s)|^2+\mathbb{E}|X_j^N(s)-X_{\frac{j}{N}}(s)|^2\right)\right]\txtd s \nonumber \\
			& \leq \frac{96}{N}L_g^2\mathcal{O}(1)\int_{0}^t\sum_{i=1}^N\mathbb{E}|X_i^N(s)-X_{\frac{i}{N}}(s)|^2\txtd s,
		\end{align*}
		where we have used that $\sum_{j=1}^N(A^N_{i,j})^2=\mathcal{O}(N)$, for all $1\leq i\leq N$. 
		Similarly, the term $\mathrm{E}_g^3$ can be bounded by using the fact that $g$ is a Lipschitz and bounded function:
		\begin{equation}
			\mathrm{E}_g^3 \leq 48 T B_g^2\left[ d_{\infty}(\eta_{A^N},\eta)\right] ^2.
		\end{equation}
		Finally, we only need to estimate $\mathrm{E}_g^2$. Note that this term can be rewritten by expanding the square as follows:
		\begin{align}
			\mathrm{E}_g^2 & = \frac{48}{N^3}\int_0^t\sum_{i=1}^N\sum_{j=1}^N\sum_{k=1}^N\mathbb{E}\left[ \left( A^N_{i,j}g(X_{\frac{i}{N}}(s),X_{\frac{j}{N}}(s))-\int_\mathbb{X}g(X_{\frac{i}{N}}(s),y)A^N_{i,j}~\mu_{\frac{j}{N},s}(\txtd y)\right)\right. \nonumber \\
			& \left.\left(A^N_{i,k} g(X_{\frac{i}{N}}(s),X_{\frac{k}{N}}(s))-\int_\mathbb{X}g(X_{\frac{i}{N}}(s),y)A^N_{i,k}~\mu_{\frac{k}{N},s}(\txtd y)\right)\right]\txtd s.
		\end{align}
		We observe that, due to the independence of the $X_{\frac{i}{N}}$ as constructed in \eqref{indep}, all terms will be 0 except when $k=j$ or $k=i$. Therefore, by making use of the boundedness of $g$ we obtain
		\begin{equation*}
			\mathrm{E}_g^2 \leq \frac{96}{N}B_g^2\int_0^t\sum_{i=1}^N\frac{1}{N^2}\sum_{j=1}^N(A^N_{i,j})^2.
		\end{equation*}
		Using that $\sum_{j=1}^N(A^N_{i,j})^2=\mathcal{O}(N)$, for all $1\leq i\leq N$, we have
		\begin{equation*}
			\mathrm{E}_g^2 \leq \frac{96}{N}B_g^2T\mathcal{O}(1).
		\end{equation*}
		The bounds for the term $E_h$ are calculated in the same way as the previous term. Putting all the estimates together, we obtain:
		\begin{align*}
			\frac{1}{N}\sum_{i=1}^{N}\mathbb{E}\|X_i^N(t)-X_{\frac{i}{N}}\|^2_{*,t} &\leq \left(16L_f^2+96(L_g^2+L_h^2)\mathcal{O}(1)\right)\int_0^t\frac{1}{N}\sum_{i=1}^{N}\mathbb{E}|X_i^N(s)-X_{\frac{i}{N}}(s)|^2~\txtd s \nonumber \\
			& + \frac{96}{N}(B_g^2+B_h^2)T\mathcal{O}(1)+\frac{16}{N}\sum_{i=1}^{N}\mathbb{E}|X_i^{0}-X_{\frac{i}{N}}^0|^2  \nonumber \\
			&
            +48 T \left(B_g^2\left[ d_{\infty}\left(\eta_{A^N}, \eta\right)\right] ^2+B_h^2\left[ d_{\infty}\left(\eta_{\hat{A}^N}, \hat{\eta}\right)\right] ^2\right).
		\end{align*}
		Gr\"onwall's inequality leads to
		\begin{align*}
			\frac{1}{N}\sum_{i=1}^{N}\mathbb{E}\|X_i^N(t)-&X_{\frac{i}{N}}\|^2_{*,t} \nonumber\leq C\frac{16}{N}\sum_{i=1}^{N}\mathbb{E}|X_i^{0}-X_{\frac{i}{N}}^0|^2 \\
			&+C \left( \frac{2}{N}(B_g^2+B_h^2)\mathcal{O}(1)+B_g^2\left[ d_{\infty}\left(\eta_{A^N}, \eta\right)\right] ^2+B_h^2\left[ d_{\infty}\left(\eta_{\hat{A}^N}, \hat{\eta}\right)\right] ^2 \right) \nonumber,
		\end{align*}
	where $C=48Te^{(16L_f^2+96(L_g^2+L_h^2)\mathcal{O}(1))T}$. 
		\end{proof}
		
		To prove Theorem \ref{teorema}, we have to demonstrate that the empirical measure $\mu_N=\frac{1}{N}\sum_{i=1}^N\delta_{X_i^N}$ converges in probability to $\bar{\mu}=\int_I \mu_u du$. 
		
		%To do this, we first prove that the empirical measure $\bar{\mu}_N=\frac{1}{N}\sum_{i=1}^N\delta_{X_{\frac{i}{N}}}$ converges in probability to $\bar{\mu}$, as the following result says
		
	\begin{proof}(of Theorem~\ref{teorema}) 
		The convergence in probability, follows from the convergence for both random variables: $U$ and the one generated by the Brownian motion. To do this, it is sufficient to establish that:
		\begin{equation}
			\lim_{N\rightarrow+\infty} E_u \times \mathbb{E}\left|\int_{\mathbb{X}} f(y)~\bar{\mu}_N(\txtd y) - \int_{\mathbb{X}} f(y)~\bar{\mu}(\txtd y)\right|^2 = 0,
		\end{equation}
		for every bounded and Lipschitz function $f$, where $E_u$ denotes the expectation related to the random variable $U$.
		We know that $\mu_{i/N,t} = \mathcal{L}(X(t)|U=i/N)$, is the probability measure $\bar{\mu}_t$ conditioned on $U=i/N$. Therefore we obtain
		\begin{equation*}
			\begin{aligned}
				&\lim_{N\rightarrow+\infty} E_u \times \mathbb{E}\left|\int_{\mathbb{X}} f(y)~\bar{\mu}_N(\txtd y) - \int_{\mathbb{X}} f(y)~\bar{\mu}(\txtd y)\right|^2 \\
				&\quad \leq \lim_{N\rightarrow+\infty} \frac{1}{N} \sum_{i=1}^N E_u \times\mathbb{E}\left|f(X_i) - f\left(X_{\frac{i}{N}}\right)\right|^2.
			\end{aligned}
		\end{equation*}
		Taking into account the Lipschitz property of f and using Proposition \ref{compare}, we can conclude
		\begin{equation*}
			\begin{aligned}
				&\lim_{N\rightarrow+\infty} E_u \times \mathbb{E}\left|\int_{\mathbb{X}} f(y)~\bar{\mu}_N(\txtd y) - \int_{\mathbb{X}} f(y)~\bar{\mu}(\txtd y)\right|^2 \\
				&\quad \leq \lim_{N\rightarrow+\infty} \frac{1}{N} \sum_{i=1}^N E_u \times \mathbb{E}\left|X_i - X_{\frac{i}{N}}\right|_{*,t}^2 = 0.
			\end{aligned}
		\end{equation*}
   This finishes the proof of Theorem \ref{teorema}.
		\end{proof}

\section{Proof of Theorem \ref{teorema2}}

 Finally, let us prove Theorem \ref{teorema2}. Unlike before, since our probability measure will now take actual input values $u\in I$, instead of considering $u$ as a random variable, we will work in the space $\tilde{\mathcal{N}}$, in which we must prove the measurability with respect to $u$. The overall proof idea of Theorem~\ref{teorema2} will follow very similar ideas as before, yet certain details that need to be modified in some proofs due to working with measurability in the variable $u$. Therefore, in this section, we will not go into detail in all the proofs but will only introduce those steps that are different from the previous ones. Let us begin by demonstrating the existence and uniqueness of the solution to \eqref{indep2}.
 
 \begin{prop}
 	Under de assumptions $H$ and $\tilde{H}$, there exist a unique solution $X_u$, for $u\in I$, to \eqref{indep2}, where $\mu_{t}=(\mathcal{L}(X_u):u\in I)\in\tilde{\mathcal{N}}$. Moreover, $\mu_{u,t}$ is a weak solution of the Vlasov-Fokker-Plank equation:
 	\begin{equation}
 		\begin{split}
 			\partial_t {\mu}_{u,t} &+ \partial_x \left({\mu}_{u,t} f(x) + {\mu}_{u,t} \int_I \int_{\mathbb{X}} g(x,y) ~\mu_{v,t}(\txtd y) \eta^u(\txtd v)\right)=\\
 			&+\frac{1}{2}\partial^2_x\left(  {\mu}_{u,t} \left[ \int_I \int_{\mathbb{X}} h(x,y) ~\mu_{v,t}(\txtd y) \hat{\eta}^u(\txtd v)\right] ^2\right), \forall u\in I .
 		\end{split}
 		\label{ito2}
 	\end{equation}
 \end{prop}

\begin{proof}
	The proof will be carried out in the same way as we did in Proposition \ref{propindep}. We will consider an operator $\tilde{\mathcal{F}}$ defined in the space $\mathcal{L}$ and seek a fixed point. To do this, let us consider the mapping $\mu \in \tilde{\mathcal{N}} \mapsto \tilde{\mathcal{F}}(\mu) \in \tilde{\mathcal{N}}$, where $\tilde{\mathcal{F}}(\mu)$ is the law associated with the solution of the equation:
	\begin{equation}
		\begin{split}
			X^\mu_u(t)=&X^\mu_u(0)+\int_{0}^{t}f(X^\mu_u(s))~\txtd s+\int_{0}^t\int_I\int_{\mathbb{X}} g(X^\mu_u(s),y)~\mu_{v,s}(\txtd y)\eta^u(\txtd v)\txtd s\\
			&+ \int_{0}^t\int_I\int_{\mathbb{X}} h(X^\mu_u(s),y)~\mu_{v,s}(\txtd y)\hat{\eta}^u(\txtd v) \txtd v_s^u.
		\end{split}
		\label{eq3}
	\end{equation}
	Note that if we have a fixed point then $\tilde{\mathcal{F}}(\mu)=\mathcal{L}(X^\mu)=\mu$, so this would prove the existence of solution. 
	First, we must prove that the operator is well-defined, meaning that there exists a solution for \eqref{eq3}, for every $\mu \in \tilde{\mathcal{N}}$, and furthermore, the solutions are in $\tilde{\mathcal{N}}$.
	To do this, let us take $\mu\in\mathcal{N}$, and let $X^0_u(t)=X_u(0)$, $\forall t\in[0,T]$ and $u\in I$. Consider the following recurrence equation:
	\begin{eqnarray}
		X^n_u(t)\hspace{-0.25cm}&=\hspace{-0.25cm}&X^{n-1}_u(0)+\int_{0}^{t}f(X^{n-1}_u(s))~\txtd s+\int_{0}^t\int_I\int_{\mathbb{X}} g(X^{n-1}_u(s),y)~\mu_{v,s}(\txtd y)\eta^u(\txtd v)\txtd s\nonumber\\
		&&+\int_{0}^t\int_I\int_{\mathbb{X}} h(X^{n-1}_u(s),y)~\mu_{v,s}(\txtd y)\hat{\eta}^u(\txtd v)\txtd B_s^u,
		\label{induction2}
	\end{eqnarray}
	where  $X^k_u(0)=X^0_u(0)$, for all $k\geq 1$. Notice that if we prove that the sequence $\left\lbrace X^n_u\right\rbrace_{n\geq1}$ is Cauchy, and moreover, for each $n\in \mathbb{N}$, $\mathcal{L}(X^n_u:u\in I)\in\tilde{\mathcal{N}}$, then we will have that the sequence has a limit, and this limit belongs to $\tilde{\mathcal{N}}$ since this space is complete with the Wasserstein distance $W_{2, t}^{\tilde{\mathcal{N}}, \infty}$.
	
	The proof that $\left\lbrace X^n_u\right\rbrace_{n\geq0}$ is Cauchy, is identical to the one carried out in Proposition \ref{propindep}, therefore, we will not develop it. Let us focus on proving that for each $n\in\mathbb{N}$, the mapping $u\in I\mapsto \mathcal{L}(X^n_u,B^u)$ is measurable. We will prove it by induction. By construction and assumptions $\tilde{H}$, measurability holds for $n=0$. Suppose it holds for $n=0, \ldots, N-1$. Let us now establish that measurability holds for $n=N$. To show this, it is sufficient to show that it is measurable for any collection of times, that is,
	$$
	I \ni u \mapsto \mathcal{L}\left({X}_u^N\left(t_1\right), B^u\left(t_1\right), \ldots, {X}_u^N\left(t_m\right), B^u\left(t_m\right)\right) \in \mathcal{P}\left(\left(\mathbb{R}^d \times \mathbb{R}^d\right)^m\right).
	$$
	is measurable for all $0 \leq t_1 \leq \cdots \leq t_m \leq T$ and $m \in \mathbb{N}$. It further suffices to demonstrate that
	$$
	I \ni u \mapsto \mathbb{E}\left[\prod_{i=1}^m\left(\alpha_i\left({X}_u^N\left(t_i\right)\right) \beta_i\left(B^u\left(t_i\right)\right)\right)\right] \in \mathbb{R},
	$$
	is measurable, for all $0 \leq t_1 \leq \cdots \leq t_m \leq T, m \in \mathbb{N}$ and bounded and continuous functions $\left\{\alpha_i, \alpha_i: i=1, \ldots, m\right\}$ on $\mathbb{R}^d$. 
Let us now establish that $X_u^N(t)$ is measurable. To achieve this, consider $X_u^{N,\delta}(t)$ as a solution to the following auxiliary process
	$$
	\begin{aligned}
		{X}_u^{N, \delta}(t)= & {X}_u^{N-1}(0)+\int_{0}^{t}f({X}_u^{N-1}\left(\left\lfloor\frac{s}{\delta}\right\rfloor \delta\right))~\txtd s\\
		&+\int_0^t \int_I \int_{\mathbb{X}} g\left({X}_u^{N-1}\left(\left\lfloor\frac{s}{\delta}\right\rfloor \delta\right), y\right)  ~\mu_{v,\left\lfloor\frac{s}{\delta}\right\rfloor \delta}(d y)\eta^u(\txtd v) \txtd  s \\
		& +\int_0^t \int_I \int_{\mathbb{X}} h\left({X}_u^{N-1}\left(\left\lfloor\frac{s}{\delta}\right\rfloor \delta\right), x\right) ~\mu_{v,\left\lfloor\frac{s}{\delta}\right\rfloor \delta}(d x) \hat{\eta}^u(\txtd v) d B^u_s,
	\end{aligned}
	$$
	where $\delta \in(0,1)$. Note that, due to the construction of the auxiliary process, for each $\delta > 0$, we can express $X_u^{N,\delta}(t)$ as a finite sum of terms that depend on 
   \begin{equation*}
   \left\{{X}_u^{N-1}(0), {X}_u^{N-1}(\delta), \ldots, {X}_u^{N-1}\left(\left\lfloor\frac{t}{\delta}\right\rfloor \delta\right)\right\}\quad \text{and}\quad  \left\{B^u_0, B^u_\delta, \ldots, B^u_{\left(\left\lfloor\frac{t}{\delta}\right\rfloor \delta\right)}\right\}.
   \end{equation*}
 Furthermore, ${X}_u^{N, \delta}(t)$ converges to ${X}_u^k(t)$ in probability as $\delta \rightarrow 0$, for each $u \in I$. So it suffices to prove that
	$$
	I \ni u \mapsto \mathbb{E}\left[\prod_{i=1}^m\left(\alpha_i\left({X}_u^{k, \delta}\left(t_i\right)\right) \beta_i\left(B^u_{t_i}\right)\right)\right] \in \mathbb{R},
	$$
	is measurable, for all $0 \leq t_1 \leq \cdots \leq t_m \leq T, m \in \mathbb{N}$ and bounded and continuous functions $\left\{\alpha_i, \beta_i: i=1, \ldots, m\right\}$ on $\mathbb{R}^d$. Fix $t \in[0, \bar{T}]$. Since the measurability holds for $N-1$, it further suffices to show that
	$$
	{X}_u^{N, \delta}(t)=\gamma\left(u, {X}_u^{N-1}, B^u\right),
	$$
	for some measurable function $\gamma: I \times \mathcal{C}([0,T]) \times \mathcal{C}([0,T]) \rightarrow \mathbb{R}$. Using that ${X}_u^{N, \delta}(t)$ is a finite sum of terms depending on $\left\{{X}_u^{N-1}(0), {X}_u^{N-1}(\delta), \ldots, {X}_u^{N-1}\left(\left\lfloor\frac{t}{\delta}\right\rfloor \delta\right)\right\}$ and $\left\{B^u_0, B^u_\delta, \ldots, B^u_{\left(\left\lfloor\frac{t}{\delta}\right\rfloor \delta\right)}\right\}$ continuously, we have that $\gamma(u, \cdot, \cdot)$ is continuous on $\mathcal{C}([0,T]) \times \mathcal{C}([0,T])$, for each $u \in I$, and that $h(\cdot, x, w)$ is measurable on $I$ for each $(x, w) \in \mathcal{C}([0,T]) \times \mathcal{C}([0,T])$. Therefore, $\gamma$ is measurable and this proves the mesurability for $n=N$.
	In summary, we have thus proven that the operator $\tilde{\mathcal{F}}$ is well-defined. Now, we need to analyze the existence of a fixed point for the operator. However, this is nothing more than proving that given two laws $\mu$ and $\nu$ in $\tilde{\mathcal{N}}$, it holds that:
	\begin{equation}
		|W^{\tilde{\mathcal{N}},\infty}_{2,t}(\mathcal{F}(\mu),\mathcal{F}(\nu))|^2 \leq C\int_{0}^{t} |W^{\tilde{\mathcal{N}},\infty}_{2,s}(\mu,\nu)|^2 ~ \txtd s,
	\end{equation}
	for a certain constant $C > 0$. The proof of this inequality is similar to the one carried out in Proposition \ref{propindep}, so we omit the details.	
	The last inequality gives the path-wise uniqueness of the solution, and also allow us to prove the existence of solution. For this purpose we will build an iterative process, as follows. Consider $\nu=(\mathcal{L}(Z_u): u\in I)$, where $Z_u(t)=X_u(0)$ for all $u\in I $ and $t\in[0,T]$. Iterating this, and using \eqref{desuninon}, we get:
	$$W^{\tilde{\mathcal{N}},\infty}_{2,T}(\tilde{\mathcal{F}}^{n+1}(\nu),\tilde{\mathcal{F}}^n(\nu))\leq \frac{C^nT^n}{n!} |W^{\tilde{\mathcal{N}},\infty}_{2,T}(\tilde{\mathcal{F}}(\nu),\nu)|.$$
	It follows that the sequence $\left\lbrace\tilde{\mathcal{F}}^n(\nu)\right\rbrace_n$ is Cauchy for $n$ large enough, where we have used that $W^{\tilde{\mathcal{N}},\infty}_{2,T}(\tilde{\mathcal{F}}(\nu),\nu)<\infty$, due to the assumptions on the initial data and the fact that the functions $f$, $g$, and $h$ are bounded.
	This sequence will have a limit since $\tilde{\mathcal{N}}$ is a complete metric space, and hence there exists $\mu=(\mathcal{L}(X_u),u\in I)\in \tilde{\mathcal{N}}$ solution of \eqref{indep}, for the initial data $X_u(0)$.
\end{proof}

Once the existence and uniqueness of the solution to \eqref{indep2} have been established, the next step is to study how these solutions depend on the graph. In other words, what properties the solutions have with respect to the DGMs: $\eta$ and $\hat{\eta}$, and with respect to the graph's heterogeneity variable $u$. For this purpose, we have the following result.

\begin{prop}\label{propdgm2}
		Given $\hat{\eta}_1, \hat{\eta}_2 \in \mathcal{BC}(I, \mathcal{M}_+(I))$, let $\mu_1$ and $\mu_2$ be the laws of the solutions of \eqref{indep2} for the DGMs $\hat{\eta}_1$ and $\hat{\eta}_2$. Then,
	\begin{equation}
		\left[ W_{2,s}^{\tilde{\mathcal{N}},\infty}(\mu^1,\mu^2)\right]^2\leq \hat{C}_1\hat{C}_2\left[ d_\infty(\hat{\eta}_1,\hat{\eta}_2)\right] ^2 e^{\hat{C}_1T},
	\end{equation}
	where $\hat{C}_1 =3\left[ 2B_g^2\|\eta\|^2+3B_h^2\|\hat{\eta}_1\|^2\right]e^{\left(3L_f^2+3\left[ 2L_g^2\|\eta\|^2+3L_h^2\|\hat{\eta}_1\|^2\right] \right) T}$ and \newline $\hat{C}_2=\frac{3B^2_hT}{\left[ 2B_g^2\|\eta\|^2+3B_h^2\|\hat{\eta}_1\|^2\right]}$.
	
	Given $\eta_1, \eta_2 \in \mathcal{BC}(I, \mathcal{M}_+(I))$, let $\mu_1$ and $\mu_2$ be the laws of the solutions of \eqref{indep2} for the DGMs $\eta_1$ and $\eta_2$. Then,
	\begin{equation}
		\left[ W_{2,s}^{\tilde{\mathcal{N}},\infty}(\mu^1,\mu^2)\right]^2\leq C_1C_2\left[ d_\infty(\eta_1,\eta_2)\right] ^2 e^{C_1T},
	\end{equation}
	where $C_1 =3\left[ 3B_g^2\|\eta_1\|^2+2B_h^2\|\hat{\eta}\|^2\right]e^{\left(3L_f^2+3\left[ 3L_g^2\|\eta_1\|^2+2L_h^2\|\hat{\eta}\|^2\right] \right) T}$ and \newline $C_2=\frac{3B^2_gT}{\left[ 3B_g^2\|\eta_1\|^2+2B_h^2\|\hat{\eta}\|^2\right]}$.
	
	Given $u_1,u_2\in I$, let $\mu_u$ be the law of the solution of \eqref{indep2}. Then,
	\begin{align*}
		\left[ W_{2,T}(\mu_{u_1},\mu_{u_2})\right] ^2&\leq C\left( \left[ W_{2}(\bar{\mu}^0_{u_1},\bar{\mu}^0_{u_2})\right] ^2\right.\\ 
		&\left.+2T\left[ B_g^2d^2_{BL}(\eta^{u_1},\eta^{u_2})+B_h^2d^2_{BL}(\hat{\eta}^{u_1},\hat{\eta}^{u_2})\right] \right),
	\end{align*}
	where $C=16e^{16\left(L^2_f+2L^2_g|\eta|^2+2L^2_h|\hat{\eta}|^2 \right)T}$.
\end{prop}
\begin{proof}
	The proof of the first two results is analogous to the one carried out in Proposition \ref{propdgm}, so we will not go into detail. Just note that since the structure of the two equations, \eqref{indep} and \eqref{indep2}, is analogous, the bounds will be the same thanks to the relationship between the measures $W^{.,2}$ and $W^{.,\infty}$ and the fact that the set over which $u$ takes variables is the interval $I=[0,1]$.
	
	Let $u_1,u_2\in I$, and let $\mu$ be the law of the solution to \eqref{indep2}. We know that for each $u\in I$, $X_u$ satisfies the equation:
	\begin{align}
		X^u(t) =& X^u(0) + \int_{0}^{t} f(X^u(s)) ~\txtd s + \int_{0}^t \int_I \int_{\mathbb{X}} g(X^u(s), y) ~\mu_{v, s}(\txtd y)\eta^u(\txtd v) \txtd s\nonumber \\
		&+ \int_{0}^t \int_I \int_{\mathbb{X}} h(X^u(s), y)~\mu_{v, s}(\txtd y)\hat{\eta}^u(\txtd v) \txtd B_s^u.
		\label{proofcm}
	\end{align}
	To remove the dependence of $B^u_s$ with respect to $u$, we will employ the coupling method (see \cite{ChaintronReview}). This method allows us to relate our stochastic process with the following system, as in the limit both systems satisfy the same dynamics:
	\begin{align}
		X^u(t) =& X^u(0) + \int_{0}^{t} f(X^u(s)) ~\txtd s + \int_{0}^t \int_I \int_{\mathbb{X}} g(X^u(s), y) ~\mu_{v, s}(\txtd y)\eta^u(\txtd v) \txtd s\nonumber \\
		&+ \int_{0}^t \int_I \int_{\mathbb{X}} h(X^u(s), y)~\mu_{v, s}(\txtd y)\hat{\eta}^u(\txtd v) \txtd B_s.
		\label{proofcm2}
	\end{align}
	Therefore, since we are interested in analyzing what happens in the limit, it will be enough to prove the property for the second system. Let us proceed to bound $\mathbb{E}|X_{u_1}(t)-X_{u_2}(t)|^2$:
	\begin{align*}
		\mathbb{E}|X_{u_1}(t)-X_{u_2}(t)|^2\leq 4\mathbb{E}|X_{u_1}(0)-X_{u_2}(0)|^2+4\mathbb{E}\int_{0}^{t} |f(X_{u_1}(s))-f(X_{u_2}(s))|^2~\txtd s\\
		+4\mathbb{E}\int_{0}^t |\int_I \int_{\mathbb{X}} g(X^{u_1}(s), y) ~\mu_{v, s}(\txtd y)\eta^{u_1}(\txtd v)-g(X^{u_2}(s), y) ~\mu_{v, s}(\txtd y)\eta^{u_2}(\txtd v)|^2 \txtd s\\
		+4\mathbb{E}\int_{0}^t |\int_I \int_{\mathbb{X}} h(X^{u_1}(s), y) ~\mu_{v, s}(\txtd y)\eta^{u_1}(\txtd v)-h(X^{u_2}(s), y) ~\mu_{v, s}(\txtd y)\hat{\eta}^{u_2}(\txtd v)|^2 \txtd s.
	\end{align*}
We have used the Hölder inequality and the properties of stochastic integrals. By adding and subtracting terms in the last two integrals, we obtain the following:
		\begin{align*}
		\mathbb{E}|X_{u_1}(t)-X_{u_2}(t&)|^2\leq 4\mathbb{E}|X_{u_1}(0)-X_{u_2}(0)|^2+4\mathbb{E}\int_{0}^{t} |f(X_{u_1}(s))-f(X_{u_2}(s))|^2~\txtd s\\
		&+8\mathbb{E}\int_{0}^t \left|\int_I \int_{\mathbb{X}} \left( g(X^{u_1}(s), y) -g(X^{u_2}(s), y)\right)~  \mu_{v, s}(\txtd y)\eta^{u_1}(\txtd v)\right|^2 \txtd s\\
		&+8\mathbb{E}\int_{0}^t \left|\int_I \int_{\mathbb{X}} g(X^{u_1}(s), y)~ \mu_{v, s}(\txtd y)\left( \eta^{u_1}(\txtd v)-\eta^{u_2}(\txtd v)\right) \right|^2 \txtd s\\
		&+8\mathbb{E}\int_{0}^t \left|\int_I \int_{\mathbb{X}} \left( h(X^{u_1}(s), y) -h(X^{u_2}(s), y)\right)  ~\mu_{v, s}(\txtd y)\hat{\eta}^{u_1}(\txtd v)\right|^2 \txtd s\\
		&+8\mathbb{E}\int_{0}^t \left|\int_I \int_{\mathbb{X}} h(X^{u_1}(s), y) ~\mu_{v, s}(\txtd y)\left( \eta^{u_1}(\txtd v)-\hat{\eta}^{u_2}(\txtd v)\right) \right|^2 \txtd s.
	\end{align*}
Using the definition of bounded Lipschitz distance for the DGM and leveraging the Lipschitz and boundedness properties of the functions $f$, $h$, and $g$, we obtain the following estimate:
	\begin{align*}
		\mathbb{E}|X_{u_1}(t)-X_{u_2}(t)|^2&\leq4\mathbb{E}|X_{u_1}(0)-X_{u_2}(0)|^2\\ &+4\left(L^2_f+2L^2_g\|\eta\|^2+2L^2_h\|\hat{\eta}\|^2 \right) \mathbb{E}\int_{0}^{t} |X_{u_1}(s)-X_{u_2}(s)|^2~\txtd s\\
		&+8T\left[ B_g^2d^2_{BL}(\eta^{u_1},\eta^{u_2})+B_h^2d^2_{BL}(\hat{\eta}^{u_1},\hat{\eta}^{u_2})\right] .
	\end{align*}
Taking the supremum in $s$ of the difference $|X_{u_1}(s) - X_{u_2}(s)|$, and using the BDG inequality, we get:
	\begin{align*}
		\mathbb{E}\|X_{u_1}-X_{u_2}\|_{*,t}^2&\leq16\mathbb{E}|X_{u_1}(0)-X_{u_2}(0)|^2\\ &+16\left(L^2_f+2L^2_g\|\eta\|^2+2L^2_h\|\hat{\eta}\|^2 \right) \mathbb{E}\int_{0}^{t} \|X_{u_1}-X_{u_2}\|_{*,s}^2~\txtd s\\
		&+32T\left[ B_g^2d^2_{BL}(\eta^{u_1},\eta^{u_2})+B_h^2d^2_{BL}(\hat{\eta}^{u_1},\hat{\eta}^{u_2})\right] .
	\end{align*}
Using the Gr\"onwall's inequality yields
	\begin{align*}
		\mathbb{E}\|X_{u_1}-X_{u_2}\|_{*,t}^2&\leq C\left( \mathbb{E}|X_{u_1}(0)-X_{u_2}(0)|^2\right.\\ 
		&\left.+2T\left[ B_g^2d^2_{BL}(\eta^{u_1},\eta^{u_2})+B_h^2d^2_{BL}(\hat{\eta}^{u_1},\hat{\eta}^{u_2})\right] \right) ,
	\end{align*}
	where $C=16e^{16\left(L^2_f+2L^2_g|\eta|^2+2L^2_h|\hat{\eta}|^2 \right)T}$. By the definition of the Wasserstein distance, we have:
	\begin{align*}
		\left[ W_{2,T}(\mu_{u_1},\mu_{u_2})\right] ^2&\leq C\left( \mathbb{E}|X_{u_1}(0)-X_{u_2}(0)|^2\right.\\ 
		&\left.+2T\left[ B_g^2d^2_{BL}(\eta^{u_1},\eta^{u_2})+B_h^2d^2_{BL}(\hat{\eta}^{u_1},\hat{\eta}^{u_2})\right] \right). 
	\end{align*}
As this inequality holds for any random variable $X_u(0)$, if we take the infimum, we get:
	\begin{align*}
		\left[ W_{2,T}(\mu_{u_1},\mu_{u_2})\right] ^2&\leq C\left( \left[ W_{2}(\bar{\mu}^0_{u_1},\bar{\mu}^0_{u_2})\right] ^2\right.\\ 
		&\left.+2T\left[ B_g^2d^2_{BL}(\eta^{u_1},\eta^{u_2})+B_h^2d^2_{BL}(\hat{\eta}^{u_1},\hat{\eta}^{u_2})\right] \right). 
	\end{align*}
Thus, we obtain the desired inequality and conclude the proof.
\end{proof}

After establishing the existence of solution for \eqref{indep2} and studying the properties of it, let us proceed to determine how close the solutions of \eqref{eq1} are to those of \eqref{indep2}.
To accomplish this, as we did in the previous section, we will compare the trajectories of solutions between equation \eqref{eq1} and equation \eqref{indep}, associating the particles $X_i$ with the particles $X_u(i)$. To facilitate this comparison, we will consider a partition of the interval $I$ defined by $I^N_i$, as introduced earlier. Within each interval, we will select a representative $u(i) = \frac{i}{N}$. Consequently, we will compare the trajectories of $X_i^N$ and $X_{\frac{i}{N}}$.
It is important to point out that each one is described by a distinct stochastic process. However, to identify both particles as similar, we make the assumption that for our choice of the random variable $u=\frac{i}{N}$, both particles exhibit the same Brownian motions, i.e., $B^i_t=B^{\frac{i}{N}}_t$.

	\begin{prop}\label{propconv2}
	Under the hypotheses of Theorem \ref{teorema2}, the following inequality
		\begin{align*}
	\frac{1}{N}\sum_{i=1}^{N}\mathbb{E}\|X_i^N(t)-&X_{\frac{i}{N}}\|^2_{*,t} \nonumber\leq C\frac{16}{N}\sum_{i=1}^{N}\mathbb{E}|X_i^{0}-X_{\frac{i}{N}}^0|^2 \\
	&+C \left( \frac{2}{N}(B_g^2+B_h^2)\mathcal{O}(1)+B_g^2\left[ d_{\infty}\left(\eta_{A^N}, \eta\right)\right] ^2+B_h^2\left[ d_{\infty}\left(\eta_{\hat{A}^N}, \hat{\eta}\right)\right] ^2 \right) \nonumber,
\end{align*}
	holds, where $C=48Te^{(16L_f^2+96(L_g^2+L_h^2)\mathcal{O}(1))T}$.
	\label{compare2}
\end{prop}
\begin{proof}
	Since the proof of this result is identical to the one carried out in the previous section, we will not go into detail about the proof.
\end{proof}

 We can now proceed to prove Theorem \ref{teorema2}.
\begin{proof}

	We want to prove that the empirical measure $\mu^N$ converges in probability to the measure $\bar{\mu}$. To do this, it is sufficient to establish that:
	\begin{equation*}
		\lim_{N\rightarrow+\infty} \mathbb{E}\left|\int_{\mathbb{X}} k(y)~{\mu}_N(\txtd y) - \int_{\mathbb{X}} k(y)~\bar{\mu}(\txtd y)\right|^2 = 0,
	\end{equation*}
	for every bounded and Lipschitz function $k$. To show that this tends to zero, let us add and subtract some term:
	\begin{align}
		\mathbb{E}\left|\int_{\mathbb{X}} k(y)~{\mu}^N(\txtd y) - \int_{\mathbb{X}} k(y)~\bar{\mu}(\txtd y)\right|^2&\leq  \mathbb{E}\left|\int_{\mathbb{X}} k(y)~{\mu}^N(\txtd y) - \frac{1}{N}\sum_{i=1}^N\int_{\mathbb{X}} k(y)\delta_{X_{\frac{i}{N}}}(y)\right|^2 \nonumber\\
		&+\mathbb{E}\left|\frac{1}{N}\sum_{i=1}^N\left( \int_{\mathbb{X}} k(y)~\delta_{X_{\frac{i}{N}}}(\txtd y) -\int_{\mathbb{X}} k(y)~\mu_{\frac{i}{N}}(\txtd y)\right) \right|^2 \nonumber \\
		&+\mathbb{E}\left|\frac{1}{N}\sum_{i=1}^N\int_{\mathbb{X}} k(y)~\mu_{\frac{i}{N}}(\txtd y) - \int_{\mathbb{X}} k(y)~\bar{\mu}(\txtd y)\right|^2.
		\label{proofineqteo2}
	\end{align}
	Let us see that each term converges to zero. If we expand the first term, we get:
	\begin{align*}
		\mathbb{E}&\left|\int_{\mathbb{X}} k(y)~{\mu}^N(\txtd y) - \frac{1}{N}\sum_{i=1}^N\int_{\mathbb{X}} k(y)~\delta_{X_{\frac{i}{N}}}(y)\right|^2 \nonumber \\
		&\leq\frac{1}{N}\sum_{i=1}^N\mathbb{E}|k(X_i)-k(X_{\frac{i}{N}})|^2.
	\end{align*}
	Using $k$ is Lipschitz and by Proposition \ref{propconv2}, we deduce that it tends to zero. Let us proceed with the second term of \eqref{proofineqteo2}.
	$$
	\mathbb{E}\left|\frac{1}{N}\sum_{i=1}^N\left( \int_{\mathbb{X}} k(y)~\delta_{X_{\frac{i}{N}}}(\txtd y) -\int_{\mathbb{X}} k(y)~\mu_{\frac{i}{N}}(\txtd y)\right) \right|^2 \leq\frac{1}{N}\sum_{i=1}^N\mathbb{E}\left|k(X_{\frac{i}{N}})-\mathbb{E}|k(X_{\frac{i}{N}})|\right|^2.
	$$
	Using that $k$ is bounded, we prove that this term converges to zero. Finally, for the last term of \eqref{proofineqteo2}, we can write it as:
	$$
	\mathbb{E}\left|\frac{1}{N}\sum_{i=1}^N\int_{\mathbb{X}} k(y)~\mu_{\frac{i}{N}}(\txtd y) - \int_I\int_{\mathbb{X}} k(y)~{\mu}_u(\txtd y)du\right|^2.
	$$
	Using Proposition \ref{propdgm2}, which establishes the continuity of $\hat{\mu}_u$ with respect to $u$, allows us to demonstrate that this term tends to zero. Consequently, $\mu^N\rightarrow\bar{\mu}$ in probability.	
\end{proof}

\section{Examples}

As demonstrated earlier, under certain conditions, the empirical measure of the stochastic process \eqref{eq1} converges to the measure $\bar{\mu}(t,x)$ that satisfies the differential equation:
\begin{equation}
		\begin{split}
			\partial_t \bar{\mu}_{t} &+ \partial_x \left(\bar{\mu}_{t} f(x) + \int_I{\mu}_{u,t} \int_I \int_{\mathbb{X}} g(x,y) ~\mu_{v,t}(\txtd y) \eta^u(\txtd v) \txtd u\right)=\\
			&+\frac{1}{2}\partial^2_x\left( \int_I {\mu}_{u,t} \left[ \int_I \int_{\mathbb{X}} h(x,y) ~\mu_{v,t}(\txtd y) \hat{\eta}^u(\txtd v)\right] ^2\txtd u\right) .
		\end{split}
    \label{example0}
\end{equation}
Now, let us very briefly explore different examples of stochastic processes and how this theorem would apply to these processes, along with the corresponding limiting probability equation.

Consider first the case where the Brownian term is only multiplied by a constant, and the graph connects all nodes with each other, i.e., the adjacency matrix $A_{ij}=1$ for $i=1,..,N$. We obtain the following stochastic process:
\begin{equation}
    \begin{aligned}
        \txtd X_i(t) &= f(X_i(t)) \txtd t + \frac{1}{N} \sum_{j=1}^{N} g(X_i, X_j) \txtd t +\sigma \txtd B_t^i.
    \end{aligned}
\end{equation}
In this case, the limit differential equation is:
\begin{equation}
		\begin{split}
			\partial_t \bar{\mu}_{t} &+ \partial_x \left(\bar{\mu}_{t} f(x) + \bar{\mu}_{t} \int_{\mathbb{X}} g(x,y) ~\mu_{t}(\txtd y) \right)=\frac{\sigma^2}{2}\partial^2_x \bar {\mu}_{t}.  
		\end{split}
\end{equation}

Now, assume that the limiting DGMs $\eta$ and $\hat{\eta}$ have sufficiently good properties, and $\eta$, $\hat{\eta} \in \mathcal{B}(I, \mathcal{M}_{+,abs}(I)) \bigcap \mathcal{C}(I, \mathcal{M}_{+,abs}(I))$, where $\mathcal{M}_{+,abs}(I)$ is the set of all finite Borel positive measures on $I$ absolutely continuous with respect to the Lebesgue measure $m$. In this case, denoting $W$ and $\hat{W}$ as the respective Radon-Nikodym derivatives of the previous DGMs, the system can be seen as a representation through graphons. In this space, the adjacency matrices are represented through graphons, and the limiting equation has the form
\begin{equation}
		\begin{split}
			\partial_t \bar{\mu}_{t} &+ \partial_x \left(\bar{\mu}_{t} f(x) + \int_I{\mu}_{u,t} \int_I \int_{\mathbb{X}} W(u,v)g(x,y) \mu_{v,t}(\txtd y) \txtd v \txtd u\right)=\\
			&+\frac{1}{2}\partial^2_x\left( \int_I {\mu}_{u,t} \left[ \int_I \int_{\mathbb{X}}\hat{W}(u,v) h(x,y) \mu_{v,t}(\txtd y) \txtd v\right] ^2du\right) .
		\end{split}
    \label{example1}
\end{equation}
This differential equation is consistent with the results obtained in other works \cite{Bayraktar,Coppini} for graphons. Yet, our result is a clear generalization as there are many DGMs that have no graphon representation as we do not assume any absolute continuity of the graph limit. For example, we can apply our result to certain sparse graphs, such as certain ring networks, where the limiting DGM would be $2\delta_u(v)$. The resulting differential equation is then:
\begin{equation}
		\begin{split}
			\partial_t \bar{\mu}_{t} &+ 2\partial_x \left(\bar{\mu}_{t} f(x) + \int_I{\mu}_{u,t} \int_{\mathbb{X}} g(x,y) ~\mu_{u,t}(\txtd y) \txtd u\right)=\\
			&+\partial^2_x\left( \int_I {\mu}_{u,t} \left[ \int_{\mathbb{X}} h(x,y) ~\mu_{u,t}(\txtd y) \right] ^2\txtd u\right) .
		\end{split}
    \label{example2}
\end{equation}
In particular, all the DGM examples discussed for the deterministic particle system case in~\cite{Kuehn} can be carried over now to the stochastic setting, which provides also a clear indication that our approach to mean-field limits via DGMs is quite robust.

\section*{Acknowledgments}
This paper (CP) has been partially supported by Grant PID2022-137228OB-I00 funded by the Spanish Ministerio de Ciencia, Innovacion y Universidades, MICIU/AEI/10.13039/501100011033 \& ``ERDF/EU A way of making Europe'', by Grant C-EXP-265-UGR23 funded by Consejeria de Universidad, Investigacion e Innovacion \& ERDF/EU Andalusia Program, and by Modeling Nature Research Unit, project QUAL21-011.

\end{document}